\pdfoutput=1
\RequirePackage{ifpdf}
\ifpdf 
\documentclass[pdftex]{sigma}
\else
\documentclass{sigma}
\fi

\usepackage{tikz-cd}

\numberwithin{equation}{section}

\newtheorem{Theorem}{Theorem}[section]
\newtheorem*{Theorem*}{Theorem}
\newtheorem{Corollary}[Theorem]{Corollary}
\newtheorem{Lemma}[Theorem]{Lemma}
\newtheorem{Proposition}[Theorem]{Proposition}
\newtheorem{Maintheorem}[Theorem]{Main Theorem}
 { \theoremstyle{definition}
\newtheorem{Definition}[Theorem]{Definition}

\newtheorem{Example}[Theorem]{Example}
\newtheorem{Remark}[Theorem]{Remark} }

\newcommand{\R}{\mathbb{R}}
\newcommand{\N}{\mathbb{N}}

\newcommand{\Z}{\mathbb{Z}}
\newcommand{\C}{\mathbb{C}}

\newcommand{\B}{\mathrm{B}}
\newcommand{\A}{\mathcal{A}}

\newcommand{\ind}{\operatorname{ind}}
\newcommand{\Cl}{\mathrm{Cl}}

\newcommand{\tr}{\mathrm{tr}}

\newcommand{\SO}{\mathrm{SO}}

\newcommand{\id}{\mathrm{id}}

\newcommand{\res}{\mathrm{res}}

\newcommand{\Aut}{\mathrm{Aut}}

\newcommand{\Spin}{\mathrm{Spin}}

\newcommand{\Imp}{\mathrm{Imp}}

\newcommand{\U}{\mathrm{U}}

\renewcommand{\O}{\mathrm{O}}

\renewcommand{\L}{\mathrm{L}}
\newcommand{\DD}{\mathrm{DD}}
\newcommand{\defeq}{\stackrel{\mathrm{def}}{=}}
\newcommand{\oclass}{\mathfrak{or}}

\begin{document}
\allowdisplaybreaks

\newcommand{\arXivNumber}{2204.00798}

\renewcommand{\PaperNumber}{020}

\FirstPageHeading

\ShortArticleName{The Clifford Algebra Bundle on Loop Space}

\ArticleName{The Clifford Algebra Bundle on Loop Space}

\Author{Matthias LUDEWIG}

\AuthorNameForHeading{M.~Ludewig}

\Address{Fakult\"at f\"ur Mathematik, Universit\"at Regensburg, 93040 Regensburg, Germany}
\Email{\href{mailto:matthias.ludewig@mathematik.uni-regensburg.de}{matthias.ludewig@mathematik.uni-regensburg.de}}

\ArticleDates{Received September 01, 2023, in final form February 22, 2024; Published online March 12, 2024}

\Abstract{We construct a Clifford algebra bundle formed from the tangent bundle of the smooth loop space of a Riemannian manifold, which is a bundle of super von Neumann algebras on the loop space. We show that this bundle is in general non-trivial, more precisely that its triviality is obstructed by the transgressions of the second Stiefel--Whitney class and the first (fractional) Pontrjagin class of the manifold.}

\Keywords{loop spaces; Clifford algebra; string geometry; von Neumann algebra bundles}

\Classification{15A66; 53C27; 46L40}

\section{Introduction}

The bundle of Clifford algebras $\Cl(X)$ constructed from the tangent bundle over a Riemannian manifold $X$ is fundamental to spin geometry.
In particular, $X$ has a spin$^c$ structure if and only if it is oriented and the Dixmier--Douady class of $\Cl(X)$ vanishes.
In particular, the spin$^c$ condition is related to (partial) triviality of $\Cl(X)$.
The purpose of this paper is to obtain similar results for the analogous bundle on the loop space $\L X$ of an \emph{oriented} Riemannian manifold, which is fundamental to string geometry (i.e., to spin geometry on the loop space).

The Riemannian metric on $X$ (which we assume to be oriented throughout) induces a natural metric on the smooth loop space.
Forming the (infinite-dimensional) algebraic Clifford algebra on each tangent space is unproblematic, but in order to make the setting amenable to analysis, we must complete these algebras in some way.

It is a fact that the infinite-dimensional Clifford algebra has a unique $C^*$ norm, and completing the fibers with respect to this norm yields a bundle of $C^*$-algebras.
However, it turns out that this bundle is always trivial, hence does not encode any information on whether the loop space satisfies a spin condition.

Instead, we consider a fiberwise completion in a suitable weak topology, which leads to a~continuous bundle $\mathcal{A}_{\L X}$ of von Neumann algebras.
The canonical grading of the Clifford algebra carries over to the von Neumann completion, and in fact, the fibers are \emph{super factors of type I}, meaning that the fibers are type I von Neumann algebras with trivial graded center.

Super factors of type I come in two stable isomorphism classes:
Those of \emph{even kind}, which are stably isomorphic to $\C$, and those of \emph{odd kind}, which are stably isomorphic to $\Cl_1$.
It turns out that if the dimension of $X$ is even, then the fibers of $\mathcal{A}_{\L X}$ are of even kind, while otherwise, they are of odd kind.

The classifying space of the automorphism group $\Aut(A)$ of a non-trivially graded, properly infinite super factor of type I turns out to be a product of Eilenberg--MacLane spaces, $B\Aut(A) \simeq K(\Z, 3) \times K(\Z_2, 1)$.
Hence bundles $\mathcal{A} \to X$ with typical fiber $A$ are classified by two characteristic classes
$
\DD(\mathcal{A}) \in H^3(X, \Z)$, $\oclass(\mathcal{A}) \in H^1(X, \Z_2)$,
which we call the \emph{Dixmier--Douady class} and the \emph{orientation class}.
The first is an analog of the class first defined in \cite{DixmierDouady}.
The second class comes from the fact that we work with bundles of \emph{super} algebras, and that all automorphisms considered are required to respect the grading.

Our main result is the calculations of these characteristic classes for the loop space Clifford algebra bundle $\mathcal{A}_{\L X}$, which in particular shows that it is non-trivial in many cases.
Explicitly, we find:

\begin{Maintheorem} \label{ThmObstructionClass}
Let $X$ be an oriented Riemannian manifold of dimension $d\geq 5$.
Then
\begin{equation*}
2\cdot \mathrm{DD}(\mathcal{A}_{\L X}) = \tau(p_1(X)),\qquad \oclass(\mathcal{A}_{\L X}) = \tau(w_2(X)),
\end{equation*}
where $\tau$ denotes transgression, and where $p_1(X)$ and $w_2(X)$ are the first Pontrjagin class, respectively the second Stiefel--Whitney class.
Moreover, if $X$ is spin, then $\mathrm{DD}(\mathcal{A}_{\L X})$ equals the transgression of the fractional Pontrjagin class $\frac{1}{2}p_1(X)$.
\end{Maintheorem}

In fact, we have a more refined version of the above theorem (see Theorem~\ref{ThmRepeatedMain}):
There is a canonical characteristic class $\mathfrak{S}(X) \in H^3(\L X, \Z)$ on the loop space such that $2 \cdot \mathfrak{S}(X) = \tau(p_1(X))$, which we call \emph{loop spin class} (see Definition~\ref{DefinitionStringClass}),
and $\DD(\mathcal{A}_{\L X})$ is expressed in terms of this class.
The interesting point is that while the fractional Pontrjagin class $\frac{1}{2}p_1(X)$ only exists when $X$ is spin, the corresponding class $\mathfrak{S}(X)$ on the loop space always exists.

The typical fiber of the bundle $\mathcal{A}_{\L X}$ is a suitable completion $A_d$ of the algebraic Clifford algebra on $H^d = L^2\big(S^1, \R^d\big)$.
The loop group $\L\SO(d)$ acts naturally on the von Neumann completion $A_d$ by Bogoliubov automorphisms, and it turns out that $\mathcal{A}_{\L X}$ can be written as an associated bundle to the looped frame bundle $\L\SO(X)$ of $X$.
Our proof of Theorem~\ref{ThmObstructionClass} is then based on the fact that the map $\Omega\SO(d) \to \Aut(A_d)$ induces an isomorphism on $\pi_k$ for $k \leq 2$; in other words, $\Aut(A_d)$ is the Postnikov truncation of $\Omega\SO(d)$.

The relation of our Clifford von Neumann algebra bundle to other objects from loop space spin geometry, such as the transgression of the Chern--Simons gerbe \cite{WaldorfSpinString} and the loop space spinor bundle \cite{ KristelWaldorf2, KristelWaldorf3,KristelWaldorf1} is best understood using the language of 2-vector bundles \cite{2vectorbundles}.
This point of view is discussed in \cite[Section~1]{LudewigSpinorBundle}.

Recall that a spin manifold $X$ is called a \emph{string manifold} if the fractional Pontrjagin class $\frac{1}{2}p_1(X)$ vanishes.
Our theorem therefore implies in particular that the loop space Clifford algebra bundle $\mathcal{A}_{\L X}$ of a string manifold $X$ is trivializable.
Hence $\mathcal{A}_{\L X}$ admits a bundle of irreducible modules, the \emph{loop space spinor bundle}, so that $\L X$ is spin.

The converse of the above statement is false, as the transgression $\tau\big(\frac{1}{2}p_1(X)\big)$ may vanish without $\frac{1}{2}p_1(X)$ being zero.
It is a general fact that such converses require the extra condition of \emph{fusion} \cite{WaldorfTransgressionI, WaldorfTransgressionII, WaldorfTransgressionIII}.
In the present context, it turns out that the bundle $\mathcal{A}_{\L X}$ is the transgression of a certain bundle of \emph{free fermion conformal nets} \cite{DouglasHenriquesModularFormsConformalNets}, which (on a spin manifold) is classified by $\frac{1}{2}p_1(X)$.
This will be discussed in future work.

\section{Bundles of super von Neumann algebras}

In this section, we define bundles of von Neumann algebras with fiber a super factor of type I, a~notion that will be explained in the next subsection.
Throughout the paper, all Hilbert spaces are assumed to be separable and all von Neumann algebras are $\sigma$-finite.

\subsection{Super factors of type I and their classification} \label{SectionSuperFactors}

A \emph{super von Neumann algebra} is a von Neumann algebra $A$ together with a normal (i.e., ultraweakly continuous) involutive $*$-automorphism $\gamma$.
Such an automorphism gives a direct sum decomposition $A = A^0 \oplus A^1$, such that $A_i \cdot A_j \subset A_{i+j}$.
The graded center of $A$ is defined as
\begin{equation*}
Z(A) = \{a \in A \mid \forall b \in A\colon  [a, b] = 0 \},
\end{equation*}
where the graded commutator is defined by $[a, b] = ab - (-1)^{|a||b|} ba$ on homogeneous elements and extends to all of $A$ by bilinearity.

\begin{Definition}[super factor]
We say that a von Neumann algebra $A$ is a \emph{super factor of type~I} if it is type I (as an ungraded von Neumann algebra) and its graded center is equal to $\C$.
We say that $A$ is \emph{properly infinite} if its even part $A^0$ is properly infinite in the usual sense.
\end{Definition}

\begin{Lemma}
\label{LemmaCenterOfSuperFactor}
 Let $A$ be a super factor.
 Then the ungraded center $Z^{\mathrm{un}}(A)$ is a graded subalgebra, which is isomorphic to either $\C$ $($trivially graded$)$ or to $\C \oplus \C$ $($with the grading operator given by swapping the two summands$)$.
\end{Lemma}

\begin{proof}
Let $a = a_0 + a_1 \in Z^{\mathrm{un}}(A)$.
Then comparing the odd and even components of $ab$ and~$ba$, for $b$ homogeneous, shows that both $a_0, a_1 \in Z^{\mathrm{un}}(A)$.
Hence $Z^{\mathrm{un}}(A)$ is a graded subalgebra.

The subalgebra $Z^{\mathrm{un}}(A)$ is an abelian von Neumann algebra, hence isomorphic to $L^\infty(X)$ for some measure space $X$.
As it is a graded subalgebra, we can write $Z^{\mathrm{un}}(A) = Z^0 \oplus Z^1$ for its graded components.
That $A$ is a super factor implies that $Z^0 = \C \cdot \mathbf{1}$, as any even element in the ungraded center is also an element of the graded center, which is trivial by assumption.
Hence $Z^{\mathrm{un}}(A) = \C\cdot \mathbf{1} \oplus Z^1$, where $Z^1$ satisfies $Z^1 \cdot Z^1 \subseteq \C\cdot \mathbf{1}$.
Suppose that $Z^{\mathrm{un}}(A) \neq \C$.
Then there exists a projection $p \neq 0, 1$ in $Z^{\mathrm{un}}(A)$.
Write $p = \lambda \mathbf{1} + p^1$ with $\lambda \in \C$, $p^1 \in Z^1(A)$.
Then $\lambda \neq 0$, as projections cannot be purely odd.
Now, for any other non-zero projection $q = \mu \mathbf{1} + q^1$ with $pq = 0$, we have
\begin{equation*}
 0 = pq = \lambda \mu \mathbf{1} + p^1q^1 + \lambda q^1 + \mu p^1.
\end{equation*}
Considering the odd part, we obtain $\lambda q^1 + \mu p^1 = 0$.
Hence (as both $\lambda, \mu \neq 0$) $q^1$ is a multiple of $p^1$, so that $q$ lies in the span of $1$ and $p$.
As von Neumann algebras are generated by their projections, this implies that $Z^{\mathrm{un}}(A)$ is two-dimensional.
An isomorphism $\Z^{\mathrm{un}}(A) \cong \C \oplus \C$ of super von Neumann algebras is then obtained by mapping $\lambda \mathbf{1} + \mu p^1$ to $(\lambda + \mu, \lambda - \mu) \in \C \oplus\C$, where $p^1 \in Z^1(A)$ is a self-adjoint element of norm one.
\end{proof}

\begin{Definition}[even/odd kind]
We say that a super factor $A$ of type I is \emph{of even kind} if its ungraded center $Z^{\mathrm{un}}(A)$ is trivial, and \emph{of odd kind} otherwise.	
\end{Definition}

\begin{Example} \label{ExampleFiniteDimensional}
 The complex Clifford algebra $\Cl_d$ on $\R^d$ is a finite-dimensional super factor of type I.
 It is of even kind when $d$ is even and of odd kind when $d$ is odd.
\end{Example}

\begin{Example} \label{ExampleOddKind}
If $H$ is a graded Hilbert space with grading operator $\Gamma$, then $\B(H)$ is a super factor of type I with grading automorphism given by conjugation with $\Gamma$.
$\B(H)$ is always of even kind, and it is non-trivially graded if and only if $\Gamma$ is non-trivial and properly infinite if and only if both $H^0$ and $H^1$ are infinite-dimensional.
\end{Example}

\begin{Example} \label{ExampleEvenKind}
A type I super factor of odd kind is obtained by taking the super tensor product~${\B(H) \otimes \Cl_1}$, where $\Cl_1$ is the complex Clifford algebra of degree one.
\end{Example}

\begin{Theorem}
\label{ThmClassificationSuperFactorI}
Let $A$ be a super factor of type I.
If $A$ is of even kind, then it is isomorphic to~$\B(H)$, with the grading operator given by conjugation with some unitary involution $\Gamma$ of $H$.
If $A$ is of odd kind, then it is isomorphic to $\B(H) \oplus \B(H)$, with grading operator given by exchanging the two summands.
\end{Theorem}

\begin{proof}
We distinguish by the two cases of Lemma~\ref{LemmaCenterOfSuperFactor}.

(i) If $Z^{\mathrm{un}}(A) = \C$, then $A$ is an ordinary type I factor, hence isomorphic to $\B(H)$.
As any automorphism of $\B(H)$ is inner, the grading automorphism $\gamma$ is given by conjugation with a~unitary $\Gamma$, which must satisfy $\Gamma^2 = z \cdot\mathbf{1}$ for some $z \in \U(1)$ as $\gamma$ is an involution.
If $w$ is some square root of $z$, then $\tilde{\Gamma} = \overline{w} \Gamma$ also implements $\gamma$ and satisfies $\tilde{\Gamma}^2 = \mathbf{1}$.

(ii) If $Z^{\mathrm{un}}(A) = \C \oplus \C$, then (as $A$ is type I), we have $A = \B(H) \oplus \B(K)$ for Hilbert spaces~$H$ and $K$.
The grading operator $\gamma$ of $A$ is then given by conjugation with a unitary
\begin{equation*}
 \Gamma = \begin{pmatrix} u & z \\ x & w \end{pmatrix}
\end{equation*}
on $H \oplus K$.
Since the restriction of $\gamma$ to $Z^{\mathrm{un}}(A)$ swaps the two factors, we have
\begin{equation*}
\begin{pmatrix} 0 & 0 \\ 0 & \mathbf{1}_K \end{pmatrix}
 = \begin{pmatrix} u & z \\ x & w \end{pmatrix}\begin{pmatrix} \mathbf{1}_H & 0 \\ 0 & 0 \end{pmatrix}\begin{pmatrix} u^* & x^* \\ z^* & w^* \end{pmatrix} =
 \begin{pmatrix} uu^* & u x^* \\ xu^* & xx^* \end{pmatrix}
 \end{equation*}
 and
 \begin{equation*}
\begin{pmatrix} \mathbf{1}_H & 0 \\ 0 & 0 \end{pmatrix}
 = \begin{pmatrix} u & z \\ x & w \end{pmatrix}\begin{pmatrix} 0 & 0 \\ 0 & \mathbf{1}_K \end{pmatrix}\begin{pmatrix} u^* & x^* \\ z^* & w^* \end{pmatrix} =
 \begin{pmatrix} zz^* & z w^* \\ wz^* & ww^* \end{pmatrix}.
 \end{equation*}
 This implies that $u$ and $w$ are zero, hence $x$ and $z$ are unitary.
 After modifying $\Gamma$ by an element of $\U(1)$ as above, we may assume that $\Gamma^2 = \mathbf{1}$.
 Then $x$ and $z$ are inverses to each other, and identifying $K$ with $H$ using $x$ gives an isomorphism of $A$ to a super factor of type I of the form claimed.
\end{proof}

\begin{Remark}
A reformulation of the previous theorem is that each super factor is isomorphic to either $\B(H)$ for some super Hilbert space $H$ or to $\B(H) \,\bar{\otimes}\, \Cl_1$ for some \emph{ungraded} Hilbert space $H$.
\end{Remark}

\begin{Remark} \label{RemarkTensorProductFactor}
By the isomorphism $\Cl_1 \,\bar{\otimes}\, \Cl_1 \cong \B\big(\C^2\big)$, the previous remark implies that if $A$ and $B$ are two super factors of type I, then their spatial super tensor product $A \,\bar{\otimes}\, B$ is again a~super factor of type I. Here the product in $A \,\bar{\otimes}\, B$ is on homogeneous elements defined by
\begin{equation*}
a_1 \otimes b_1 \cdot a_2 \otimes b_2 \defeq (-1)^{|a_2||b_1|} a_1 a_2 \otimes b_1 b_2.
\end{equation*}
\end{Remark}

\begin{Remark}
The isomorphism classification of super factors of type I is complicated by the fact that some Hilbert spaces involved may be finite-dimensional.
In particular, in the case that~$A = \B(H)$ for some super Hilbert space $H$, there is a variety of possibilities, as the even and the odd part of $H$ may have different dimensions.

Calling type I super factors $A$, $B$ \emph{stably isomorphic} when $A \,\bar{\otimes}\, \B(H) \cong B \,\bar{\otimes}\, \B(H)$ for some super Hilbert space $H$, the set of equivalence classes forms a group (isomorphic to $\Z_2$ and generated by $\Cl_1$), which is an infinite-dimensional version of the \emph{graded Brauer group}\footnote{This is the group of equivalence classes of finite-dimensional central simple algebras.
We remark that over $\C$, such algebras are the same thing as a finite-dimensional von Neumann algebras.} considered in \cite{WallGradedBrauer}.
\end{Remark}

\subsection{The automorphism group of super factors of type I}

We will now calculate the automorphism group (i.e., the group of grading-preserving $*$-automor\-phisms) of a non-trivially graded, properly infinite super factor $A$ of type I.\footnote{We remark that $*$-automorphisms of a von Neumann algebra are automatically normal, i.e., ultraweakly continuous.}
To this end, let $H$ be a~Hilbert space and define involutions $\Gamma_{\mathrm{ev}}$ and $\Gamma_{\mathrm{odd}}$ on $H\oplus H$ by
\begin{equation} \label{GammaMatrices}
\Gamma_{\mathrm{ev}}  \stackrel{\mathrm{def}}{=}  \begin{pmatrix} 1 & \hphantom{-}0 \\ 0 & -1 \end{pmatrix}, \qquad
\Gamma_{\mathrm{odd}}  \stackrel{\mathrm{def}}{=}  \begin{pmatrix} 0 & 1 \\ 1 & 0 \end{pmatrix}.
\end{equation}
Set moreover
\begin{equation*}
 A_{\mathrm{ev}} = \B(H \oplus H), \qquad A_{\mathrm{odd}} = \B(H) \oplus \B(H),
\end{equation*}
with grading automorphisms given by conjugation with $\Gamma_{\mathrm{ev}}$, respectively $\Gamma_{\mathrm{odd}}$.
Then $A_{\mathrm{ev}}$ is of even kind, while $A_{\mathrm{odd}}$ is of odd kind.
One can also show that any non-trivially graded, properly infinite super factor of type I is isomorphic to either $A_{\mathrm{ev}}$ or $A_{\mathrm{odd}}$, assuming that the Hilbert space $H$ is infinite-dimensional.

\begin{Proposition} \label{PropositionAutomorphismGroups}
The automorphism groups of $A_{\mathrm{ev}}$ and $A_{\mathrm{odd}}$ are given as follows.
\begin{enumerate}\itemsep=0pt
\item[$(i)$] $\Aut(A_{\mathrm{ev}}) \cong \mathrm{P}(\U(H) \times \U(H)) \rtimes \Z_2$, where $\Z_2$ acts by permuting the factors.
\item[$(ii)$] $\Aut(A_{\mathrm{odd}}) \cong \mathrm{PU}(H) \times \Z_2$.
\end{enumerate}
\end{Proposition}

Here the letter $\mathrm{P}$ denotes the corresponding projective group, i.e., the quotient by $\U(1)$ (where $\U(1)$ acts diagonally on $\U(H) \times \U(H)$).

\begin{proof}
(i) It is well known that the group of (not necessarily grading-preserving) normal $*$-automorphisms of $A_{\mathrm{ev}}$ is $\mathrm{PU}(H \oplus H)$, the projective unitary group of $H \oplus H$, which acts on~$A_{\mathrm{ev}}$ by conjugation.
That the conjugation with a unitary $U \in \U(H\oplus H)$ is grading preserving is equivalent to the requirement
\begin{equation*}
 \Gamma_{\mathrm{ev}} U a U^* \Gamma_{\mathrm{ev}} = U \Gamma_{\mathrm{ev}} a \Gamma_{\mathrm{ev}} U^*, \qquad \forall a \in A_{\mathrm{ev}}.
\end{equation*}
Hence $\Gamma_{\mathrm{ev}}U^*\Gamma_{\mathrm{ev}} U$ is in the (ungraded) center of $A_{\mathrm{ev}}$, which consists only of multiples of the identity.
Consequently, there exists $\lambda \in \U(1)$ such that $\Gamma_{\mathrm{ev}} U\Gamma_{\mathrm{ev}} = \lambda U$.
Writing
\begin{equation*}
 U = \begin{pmatrix} u & z \\ x & v \end{pmatrix},
 \end{equation*}
 we get
 \begin{equation*}
 \begin{pmatrix} u & z \\ x & v \end{pmatrix} = \lambda\begin{pmatrix} u & -z \\ -x & v \end{pmatrix},
\end{equation*}
which implies that $\lambda \in \{\pm 1\}$ and, moreover, this implies that either $x = z = 0$ (when $\lambda = 1$) or $u = v = 0$ (when $\lambda = -1$).
Hence the non-zero entries must be unitary, and we obtain
\begin{equation*}	
\Aut(A_{\mathrm{ev}}) = G_{\mathrm{ev}}/\U(1),
\end{equation*}
where
\begin{equation} \label{DefGev}
 G_{\mathrm{ev}} = \left\{ \left.\begin{pmatrix} u & 0 \\ 0 & v \end{pmatrix} \right| u, v \in \U(H)\right\} \cup \left\{ \left.\begin{pmatrix} 0 & z \\ x & 0 \end{pmatrix} \right| x, z \in \U(H)\right\}.
\end{equation}
There is an obvious short exact sequence
\begin{equation*}
\mathrm{P}(\U(H) \times \U(H)) \to \Aut(A_{\mathrm{ev}}) \to \Z_2,
\end{equation*}
which is split by sending the generator of $\Z_2$ to the operator $\Gamma_{\mathrm{odd}}$ defined in \eqref{GammaMatrices}.
This realizes~$\Aut(A_{\mathrm{ev}})$ as a semidirect product of $\mathrm{P}(\U(H) \times \U(H))$ with $\Z_2$.

(ii) It is straightforward to see that the group of not necessarily grading-preserving automorphisms of $A_{\mathrm{odd}}$ is precisely $\Aut(A_{\mathrm{ev}})$.
The additional requirement that such an automorphism preserves the grading operator means that $U \Gamma_{\mathrm{odd}} a \Gamma_{\mathrm{odd}} U^* = \Gamma_{\mathrm{odd}} U a U^* \Gamma_{\mathrm{odd}}$ for all $a \in A_{\mathrm{odd}}$, where $U$ is a representing unitary.
This is the case if and only if and only if $\Gamma_{\mathrm{odd}} U \Gamma_{\mathrm{odd}} = \lambda U$ for some $\lambda$ in the (ungraded) center of $A_{\mathrm{odd}}$, which in this case is generated by the identity operator and $\Gamma_{\mathrm{ev}}$.
As $U$ is either diagonal or off-diagonal, this means that
 \begin{equation*}
 	\begin{pmatrix} u & 0 \\ 0 & v \end{pmatrix} = \begin{pmatrix} \lambda v & 0 \\ 0 & \mu u \end{pmatrix}, \qquad \text{respectively} \qquad \begin{pmatrix} 0 & z \\ x & 0 \end{pmatrix} = \begin{pmatrix} 0 & \lambda x \\ \mu z& 0 \end{pmatrix},
 \end{equation*}
 for some $\lambda, \mu \in \C$.
 So $u = \lambda v$, respectively $z = \lambda x$.
 Hence
\begin{equation*} 
 \Aut(A_{\mathrm{odd}}) = G_{\mathrm{odd}} / \U(1),
\end{equation*}
where
\begin{equation} \label{DefGodd}
 G_{\mathrm{odd}} = \left\{ \left.\begin{pmatrix} u & 0 \\ 0 & u \end{pmatrix} \right| u \in \U(H)\right\} \cup \left\{ \left.\begin{pmatrix} 0 & u \\ u & 0 \end{pmatrix} \right| u \in \U(H)\right\}.
\end{equation}
But clearly, $G_{\mathrm{odd}} \cong \U(H) \times \Z_2$, hence $\Aut(A_{\mathrm{odd}}) \cong \mathrm{PU}(H) \times \Z_2$.
\end{proof}

The automorphism group $\Aut(A)$ of a von Neumann algebra $A$ will always be endowed with Haagerup's u-topology (see \cite[Section~3]{HaagerupStandardForm}).
Under the identification of Proposition~\ref{PropositionAutomorphismGroups}, this coincides with the (quotient of the) strong topology on $G_{\mathrm{ev}}/\U(1)$, respectively $G_{\mathrm{odd}}/\U(1)$ \cite[Corollary~3.8]{HaagerupStandardForm}.

\begin{Remark} \label{RemarkComponentsAutA}
The subgroups $G_{\mathrm{ev}}$ and $G_{\mathrm{odd}}$ of $\U(H \oplus H)$ defined in \eqref{DefGev} and \eqref{DefGodd} both have two connected components.
It is clear that the continuous group homomorphism $i$ from $G_{\mathrm{ev}/\mathrm{odd}}$ to $\Aut(A)$ (sending a unitary $U$ to the automorphism given by conjugation with $U$) induces an isomorphism on $\pi_0$.
\end{Remark}

\begin{Remark}
Observe that $G_{\mathrm{ev}}$ equals the set of homogeneous unitaries inside $A_{\mathrm{ev}}$.
It follows that all automorphisms of $A_{\mathrm{ev}}$ are inner.
As $A_{\mathrm{odd}} \cap G_{\mathrm{odd}} = (G_{\mathrm{odd}})_0$, the identity component of~$G_{\mathrm{odd}}$, only the automorphisms in the identity component $\Aut(A_{\mathrm{odd}})_0$ are inner.
\end{Remark}

Observe that the automorphism group $\Aut(A_{\mathrm{odd}})$ is a subgroup of the automorphism group $\Aut(A_{\mathrm{ev}})$.
We will need the following lemma.

\begin{Lemma} \label{WeakHomotopyEquivalence}
If $H$ is infinite-dimensional, then the inclusion $\Aut(A_{\mathrm{odd}}) \to \Aut(A_{\mathrm{ev}})$ is a weak homotopy equivalence.
\end{Lemma}

\begin{proof}
Consider the commutative diagram
\begin{equation*}
\begin{tikzcd}
 0 \ar[r] & \U(1) \ar[r] \ar[d, equal] & G_{\mathrm{odd}} \ar[d] \ar[r] & G_{\mathrm{odd}}/\U(1) = \Aut(A_{\mathrm{odd}})\ar[r] \ar[d] & 0 \\
 0 \ar[r] & \U(1) \ar[r] & G_{\mathrm{ev}} \ar[r] & G_{\mathrm{ev}}/\U(1) = \Aut(A_{\mathrm{ev}}) \ar[r] & 0
\end{tikzcd}
\end{equation*}
with exact rows.
It is clear that the inclusion $G_{\mathrm{odd}} \to G_{\mathrm{ev}}$ is a weak homotopy equivalence, as it is on $\pi_0$ by inspection, and the connected components of $G_{\mathrm{ev}}$ and $G_{\mathrm{odd}}$ are contractible.
We obtain that the first two vertical maps induce isomorphisms on $\pi_k$ for all $k$, hence so must the third.
\end{proof}

\begin{Corollary} \label{CorollaryHomotopyType}
The classifying space of the automorphism group of a non-trivially graded, properly infinite super factor $A$ of type I has the homotopy type
\begin{equation*}
B\Aut(A) \simeq K(\Z, 3) \times K(\Z_2, 1),
\end{equation*}
and the map induced by the inclusion $\Aut(A)_0 \to \Aut(A)$ of the identity component is trivial in the second component.
\end{Corollary}

\begin{proof}
If $A$ is of odd kind, then the result follows from the isomorphism ${\Aut(A) \cong \mathrm{PU}(H)\times \Z_2}$ (Proposition~\ref{PropositionAutomorphismGroups}\,(ii)) and the fact that $\mathrm{PU}(H)$ is a~$K(\Z, 2)$, hence its classifying space is a~$K(\Z, 3)$.
If $A$ is of even kind, its automorphism group is homotopy equivalent (as a group) to the automorphism of an odd factor, i.e., to $\mathrm{PU}(H) \times \Z_2$, by Lemma~\ref{WeakHomotopyEquivalence}.
This induces a weak homotopy equivalence between the classifying spaces.
\end{proof}

In particular, we obtain that the homotopy groups of $\Aut(A)$, for $A$ a non-trivially graded, properly infinite super factor of type I, are given by
\begin{equation} \label{HomotopyGroupsAutA}
 \pi_k(\Aut(A)) = \begin{cases} \Z_2, & k=0, \\ \Z, & k=2, \\ 0, & \text{otherwise}. \end{cases}
\end{equation}

\subsection{Bundles of graded type I factors} \label{SectionBundles}

Let $S$ be a topological space and let $\A_s$, $s \in S$ be a collection of super von Neumann algebras.
For a subset $U \subset S$, we write $\A|_U$ for the disjoint union of all $\A_s$, $s \in U$.
By a \emph{local trivialization}, we mean a map $\varphi\colon \A|_U \to U \times A$, $A$ a super von Neumann algebra, that restricts to grading preserving $*$-homomorphisms $\A_s$ to $\{s\} \times A$ for each $s \in U$.
Two local trivializations are called \emph{compatible} if the corresponding transition function is continuous as a map $U \cap V \to \Aut(A)$ (endowed with Haagerup's u-topology).

\begin{Definition} \label{DefinitionBundlevN}
A collection $\A$ as above together with a maximal compatible collection of transition functions is called a (continuous) \emph{von Neumann algebra bundle with typical fiber} $A$.
\end{Definition}

If $P$ is a principal $\Aut(A)$-bundle over $S$, then the associated bundle
\begin{equation*}
\A = P \times_{\Aut(A)} A
\end{equation*}
 is a von Neumann algebra bundle with typical fiber $A$ in the sense of Definition~\ref{DefinitionBundlevN}.
It follows that isomorphism classes of super von Neumann algebra bundles over $S$ 
correspond bijectively to isomorphism classes of principal $\Aut(A)$-bundles (see, e.g., \cite[Section~4]{PennigWStarBundles}).
As (for sufficiently nice spaces $S$) such bundles are in turn classified by maps to the classifying space~$B\Aut(A)$, we obtain the following result.

\begin{Proposition} \label{PropositionClassifyingWStar}
 Isomorphism classes of von Neumann algebra bundles with typical fiber $A$ over a paracompact Hausdorff space $S$ are in bijection with homotopy classes of maps $S \to B\Aut(A)$.
\end{Proposition}

Cohomology classes over $B\Aut(A)$ provide characteristic classes for bundles with typical fiber $A$ via pullback.
If $A$ is a non-trivially graded, properly infinite super factor of type I,
Corollary~\ref{CorollaryHomotopyType} implies that
\begin{equation*}
 H^3(B\Aut(A), \Z) \cong \Z, \qquad H^1(B\Aut(A), \Z_2) \cong \Z_2.
\end{equation*}
Denote by $\oclass$ the generator of $H^1(B\Aut(A), \Z_2)$.
There is also a preferred generator $\DD$ of~$H^3(B\Aut(A), \Z)$, defined as the transgression of the first Chern class of the canonical $\U(1)$-bundle over $\Aut(A)$ (which over the identity component is $\U\big(A^0\big) \to \Aut(A)_0$);
see Appendix~\ref{AppendixSign} for more details.

\begin{Definition}[characteristic classes] \label{DefinitionCharacteristicClasses}
Let $A$ be a non-trivially graded, properly infinite super factor of type I and let $\A$ be a von Neumann algebra bundle with typical fiber $A$ and classifying map $f\colon S \to B\Aut(A)$.
The characteristic classes
\begin{equation*}
 \DD(\A) \defeq f^* \DD \in H^3(S, \Z), \qquad \oclass(\A) \defeq f^*\oclass \in H^1(S, \Z_2)
\end{equation*}
will be called the \emph{Dixmier--Douady class}, respectively the \emph{orientation class} of $\mathcal{A}$.
\end{Definition}

The terminology for $\DD(\A)$ follows that for the analogous class for bundles with typical fiber the algebra of compact operators, first defined by Dixmier and Douady \cite{DixmierDouady}.

By Corollary~\ref{CorollaryHomotopyType}, for a non-trivially graded, properly infinite super factor of type I, $B\Aut(A)$ is a product of Eilenberg--MacLane spaces.
As these are classifying spaces for cohomology, we obtain the following result.

\begin{Proposition}
Suppose that $S$ has the homotopy type of a CW complex and let $\A$ be a von Neumann algebra bundle with typical fiber $A$ over $S$, where $A$ is a super factor of type I.
Then~$\A$ is trivializable if and only if the characteristic classes $\DD(\A)$ and $\oclass(\A)$ are zero.
\end{Proposition}

\begin{Remark} \label{RemarkCheckPicture}
For CW-complexes $S$, the characteristic classes of Definition~\ref{DefinitionCharacteristicClasses} can be conveniently described using \v{C}ech cohomology, as follows.
Over a suitable open cover $\{O_\alpha\}_{\alpha \in I}$, we can choose super Hilbert spaces $H_\alpha$ and grading-preserving $*$-isomorphisms $\phi_\alpha\colon \mathcal{A}|_{O_\alpha} \to \B(H_\alpha)$.
Over two-fold intersections, we can choose families $U_{\alpha\beta}\colon (O_\alpha \cap O_\beta) \times H_\alpha \to H_\beta$ of homogeneous unitaries such that $\phi_\alpha \circ \phi_\beta^{-1}$ is given by conjugation with $U_{\alpha\beta}$.
A $\Z_2$-valued \v{C}ech 1-cocycle is obtained by defining $\varepsilon_{\alpha\beta} = \{\pm 1\}$, depending on whether $U_{\alpha\beta}$ is grading preserving or grading reversing.
Over $O_\alpha \cap O_\beta \cap O_\gamma$, we have
\begin{equation*}
U_{\gamma \alpha} U_{\beta\gamma} U_{\alpha\beta} = \lambda_{\alpha\beta\gamma} \cdot \mathrm{id}_\alpha
\end{equation*}
for some function $\lambda_{\alpha\beta\gamma} \colon O_\alpha \cap O_\beta \cap O_\gamma \to \U(1)$, so we obtain a $\U(1)$-valued \v{C}ech cochain $\{\lambda_{\alpha\beta\gamma}\}_{\alpha\beta\gamma\in I}$.
One checks that this cochain is closed with respect to the \v{C}ech coboundary, and that a cochain obtained from another choice of unitaries $\{\tilde{U}_{\alpha\beta}\}_{\alpha \beta \in I}$ differs from this cochain by a coboundary.
Hence we obtain a well-defined element in $\check{H}^2(S, \U(1))$.
Under the Bockstein homomorphism for the sequence $\Z \to \R \to \U(1)$, this element corresponds to the Dixmier--Douady class.
\end{Remark}

The Dixmier--Douady class can also be defined for bundles $\mathcal{A}$ with typical fiber a \emph{finite-dimensional} non-trivially graded super factor $A$ of type I.
One way to do this is through the \v{C}ech picture from Remark~\ref{RemarkCheckPicture}.
A second, equivalent, way is via stabilization: We replace $\mathcal{A}$ by the bundle $\mathcal{A} \,\bar{\otimes}\, \B(H)$ for some infinite-dimensional super Hilbert space $H$ such that both $H^0$ and $H^1$ are infinite-dimensional and take the characteristic classes in the sense of Definition~\ref{DefinitionCharacteristicClasses} of this bundle.

\begin{Remark}
Let $A = A_{\mathrm{ev}}$ or $A_{\mathrm{odd}}$ be one of the super factors from Section~\ref{SectionSuperFactors}, constructed in terms of a \emph{finite-dimensional} Hilbert $H$.
In this case, it turns out that
\[H^1(B\Aut(A), \Z_2) = \Z_2\qquad and \qquad H^3(B\Aut(A), \Z) = \Z_n,\]
 where $n = \dim(H)$.
We get a group homomorphism $\Aut(A) \to \Aut(A \bar{\otimes} \B(H^\prime))$ (where $H^\prime$ is an infinite-dimensional Hilbert space as above) by sending $\varphi \to \varphi \otimes \mathrm{id}$, and one can show that pullback along this homomorphism induces an isomorphism on $H^1$ and is reduction mod $n$ on~$H^3$.
If now $\mathcal{A}$ is a bundle over $S$ with typical fiber $A$, the classifying map for $\mathcal{A} \bar{\otimes} \B(H^\prime)$ factors through $B\Aut(A)$, which shows that the Dixmier--Douady class of such a bundle is $n$-torsion.
\end{Remark}

\begin{Remark}
A bundle $\mathcal{A}$ with typical fiber a properly infinite super factor of type I is trivial if and only if $\mathcal{A} = \B(\mathcal{H})$ for some bundle $\mathcal{H}$ of super Hilbert spaces over $S$.
For the class~$\DD(\mathcal{A})$ to be zero it suffices that $\mathcal{H}$ exists as a bundle of Hilbert spaces which is not necessarily globally graded (i.e., $\mathcal{H}$ does split locally into two subbundles but not necessarily globally).
The class~$\oclass(\mathcal{A})$ is zero if $\mathcal{H}$ splits into the direct sum of two subbundles, but may only be a projective bundle.
In the ungraded setting, projective bundles of Hilbert spaces were discussed in \cite{AtiyahSegalTwistedKtheory}.
\end{Remark}

Let $\mathcal{A}$ and $\mathcal{B}$ be two von Neumann algebra bundles over a space $S$ with typical fibers type I super factors $A$ and $B$.
Then the fiberwise spatial super tensor product $\mathcal{A} \,\bar{\otimes}\, \mathcal{B}$ has a canonical structure of a von Neumann algebra bundle with typical fiber $A \,\bar{\otimes}\, B$, which is again a type I super factor (see Remark~\ref{RemarkTensorProductFactor}).
The proof of the following result is analogous to that of Lemma~9 (respectively, Lemma~4) in \cite{DonovanKaroubi}.

\begin{Proposition} \label{PropTensorProductDD}
We have
\begin{equation*}
\oclass(\mathcal{A} \bar{\otimes}\mathcal{B}) = \oclass(\mathcal{A}) + \oclass(\mathcal{B}), \qquad \DD(\mathcal{A} \bar{\otimes}  \mathcal{B}) = \DD(\mathcal{A}) + \DD(\mathcal{B}) + \beta(\oclass(\mathcal{A}) \smile \oclass(\mathcal{B})),
\end{equation*}
where $\beta\colon H^2(S, \Z_2) \to H^3(S, \Z)$ is the Bockstein homomorphism.
\end{Proposition}

\section{Clifford von Neumann algebras}

In this section, we explain the construction of the von Neumann algebra completion of the algebraic Clifford algebra, given the choice of an equivalence class of (sub-)Lagrangians, and we recall the action by restricted orthogonal transformations on this algebra.
General references for the theory of Clifford (and, closely related, CAR) algebras and Fock representations are, e.g., \cite{Araki1,PlymenRobinson,PressleySegal,PreviatoSpera,ShaleStinespring,SperaWurzbacherSpinors,WurzbacherFermionicSecondQuantization}.
The completion of the Clifford algebra to a hyperfinite factor of type~II$_1$ is considered in \cite[Section~1.3]{PlymenRobinson}, but the subsequent discussion of a completion to a~factor of type I$_\infty$ seems to be new.

\subsection{Clifford algebras and Fock spaces} \label{SectionCliffordAlgebras}

Let $H$ be a real Hilbert space and denote its complexification by $H_\C$.
Let $\Cl_{\mathrm{alg}}(H)$ be the algebraic Clifford algebra, generated by elements of $H_\C$, subject to the relation
\begin{equation*}
 v \cdot w + w \cdot v = - 2 \langle \overline{v}, w\rangle.
\end{equation*}
In order to make the situation accessible to analysis, we have to complete $\Cl_{\mathrm{alg}}(H)$ to $C^*$-algebra, using the $*$-operation given by
\begin{equation*}
 (v_1\cdots v_n)^* = \overline{v}_n \cdots \overline{v}_1, \qquad v, w \in H_\C.
\end{equation*}
In fact, any $*$-representation of $\Cl_{\mathrm{alg}}(H)$ induces the same $C^*$-norm on $\Cl_{\mathrm{alg}}(H)$ \cite[Proposition~1]{ShaleStinespring}, and it follows that the Clifford algebra has a unique norm-completion $\Cl(H)$ to a~$C^*$-algebra, which turns out to be isomorphic to the infinite tensor product $M_2(\C)^{\otimes \infty}$ \cite{Araki1}.
It is moreover a Real $C^*$-algebra, as the complex conjugation of $H$ extends to an anti-linear $*$-automorphism of $\Cl(H)$.

The situation is quite different when we ask for completions of $\Cl_{\mathrm{alg}}(H)$ to a von Neumann algebra.
Such a completion can be obtained by the choice of a \emph{Lagrangian}, which is a complex subspace $L \subset H_\C$ such that $L^\perp = \overline{L}$.
The Clifford algebra $\Cl_{\mathrm{alg}}(H)$ then has a natural representation $\pi_L$ on the \emph{Fock space} $\mathcal{F}_L = \Lambda L$ (the Hilbert space exterior power of $L$) where elements~$v \in L \subset H_\C$ act by exterior multiplication and elements $v \in \overline{L} \subset H_\C$ act by contraction. We can therefore take the von Neumann completion
\[\Cl_L(H) := \pi_L(\Cl_{\mathrm{alg}}(H))^{\prime\prime}\]
 in the space~$\B(\mathcal{F}_L)$ of bounded operators on $\mathcal{F}_L$.
The Fock space is a super Hilbert space with its even/odd grading and the Fock representation is a graded representation, hence $\Cl_L(H)$ is naturally a super von Neumann algebra.
However, the real structure on $\Cl_{\mathrm{alg}}(H)$ does \emph{not} extend to a real structure on $\Cl_L(H)$ if $H$ is infinite-dimensional. It follows from irreducibility of the Fock representation $\mathcal{F}_L$ \cite[Theorem~2.4.2]{PlymenRobinson} that any bounded operator on $\mathcal{F}_L$ commuting with the Clifford action is scalar, hence
\begin{equation}
\label{IdentificationWithBH}
 \Cl_L(H)  \stackrel{\mathrm{def}}{=}  \pi_L(\Cl_{\mathrm{alg}}(H))^{\prime\prime} = \B(\mathcal{F}_L),
\end{equation}
is a super factor of type I, of even kind.

The choice of Lagrangian can be partially eliminated as follows:
Two Lagrangians $L_1$, $L_2 \subset H$ are \emph{equivalent} if the difference $P_{L_1} - P_{L_2}$ is a Hilbert--Schmidt operator, where $P_{L_i}$ denotes the orthogonal projection onto $L_i$.
By the \emph{Segal--Shale equivalence criterion}, two Fock representation~$\pi_{L_1}$ and $\pi_{L_2}$ are unitarily equivalent if and only if $L_1$ and $L_2$ are equivalent \cite[Theorem~3.4.1]{PlymenRobinson}.
Moreover, the unitary implementing the equivalence is grading-preserving if and only if $\dim\big(\overline{L}_1 \cap L_2\big)$ is even (see \cite[Theorem~1.22]{PratWaldron} and \cite[Theorem~3.5.1]{PlymenRobinson}).
We denote the equivalence class of a Lagrangian $L$ by $[L]$.

For any $L^\prime \in [L]$, $\Cl_{L^\prime}(H)$ is a completion of $\Cl_{\mathrm{alg}}(H)$ with respect to the pullback via $\pi_{L^\prime}$ of the weak operator topology on $\B(\mathcal{F}_{L^\prime})$.
However, as all representations $\pi_{L^\prime}$, $L^\prime \in [L]$, are equivalent, all these topologies coincide. As any two completions of a topological vector space are canonically isomorphic, we obtain a universal Clifford algebra associated to an equivalence class of Lagrangians.

\begin{Remark}
An explicit description of this von Neumann algebra is as the set of equivalence classes $(a_{L^\prime})_{L^\prime \in [L]}$ with $a_{L^\prime} \in \Cl_{L^\prime}(H)$, which are related by $\varphi_{L^\prime, L^{\prime\prime}}(a_{L^\prime}) = a_{L^{\prime\prime}}$, with $\varphi_{L^\prime, L^{\prime\prime}}$ the unique normal $*$-homomorphism $\B(\mathcal{F}_{L^\prime}) \to \B(\mathcal{F}_{L^{\prime\prime}})$ sending $\pi_{L^\prime}(v)$ to $\pi_{L^{\prime\prime}}(v)$ for every $v \in H$.
Another approach is to take the abstract completion of $\Cl_{\mathrm{alg}}(H)$ with respect to the ultraweak topology induced by any $\pi_L$, defined in terms of equivalence classes of Cauchy nets.
\end{Remark}

A \emph{sub-Lagrangian} is a closed subspace $L \subset H_\C$ such that $\overline{L} \subseteq L^\perp$ and such that $L + \overline{L}$ has finite codimension in $H_\C$.
Again, two sub-Lagrangians $L_1$, $L_2$ are called \emph{equivalent} if $P_{L_1}- P_{L_2}$ is a Hilbert--Schmidt operator.
Associated to an equivalence class of sub-Lagrangians, we still have a canonical completion of $\Cl_{\mathrm{alg}}(H)$, constructed as follows.
First we need the following lemma.

\begin{Lemma} \label{LemmaParitySubLagrangian}
If $L_1$, $L_2$ are two equivalent sub-Lagrangians, then $\dim\big(L_1 \oplus \overline{L}_1\big)^\perp$ and $\dim\big(L_2 \oplus \overline{L}_2\big)^\perp$ have the same parity.
\end{Lemma}

\begin{proof}
Consider the operators \smash{$J_i^\C = i(P_{L_i} - P_{\overline{L}_i})$} on $H_\C$.
Since \smash{$\overline{P}_{L_i} = P_{\overline{L}_i}$}, the operators $J_i^\C$ commute with complex conjugation, hence are the complex linear extension of operators on~$H$ denoted by $J_i$.
Observe that these operators are skew-adjoint, hence by \cite{AtiyahSkew}, they have a~well-defined index $\ind(J_i) = \dim\ker(J_i) \mod 2 \in \Z_2$.
Observe that \smash{$\ker(J_i)\otimes_\R \C = \ker\big(J_i^\C\big) = \big(L_i \oplus \overline{L}_i\big)^\perp$}, hence \smash{$\ind(J_i) = \dim\big(L_i \oplus \overline{L}_i\big)^\perp \mod 2$}.
However, by the assumption that $L_1$ and $L_2$ are equivalent, the difference $J_1 - J_2$ is a Hilbert--Schmidt operator, in particular compact.
This implies that $\ind(J_1) = \ind(J_2)$, so the lemma follows.
\end{proof}

For a sub-Lagrangian $L$, consider the complex subspace $K = \big(L \oplus \overline{L}\big)^\perp$ of $H_\C$.
The construction of the desired completion of $\Cl_{\mathrm{alg}}(H)$ depends on the dimension of $K$.
\begin{enumerate}\itemsep=0pt
\item[(i)] If $K$ is even-dimensional, we can find a Lagrangian $F \subset K$, and $L + F \in [L]$ is a Lagrangian in $H_\C$.
This yields the completion $\Cl_{L+F}(H)$ of $\Cl_{\mathrm{alg}}(H)$.
If $L^\prime$ is a sub-Lagrangian~equivalent to $L$, then by Lemma~\ref{LemmaParitySubLagrangian}, \smash{$K^\prime = \big(L^\prime \oplus \overline{L}^\prime\big)^\perp$} is still even-dimensional, and for any Lagrangian $F^\prime \subset K^\prime$, $L^\prime + F^\prime$ is equivalent to $L+F$.
Hence $\Cl_{L^\prime + F^\prime}(H)$ is canonically isomorphic to $\Cl_{L+F}(H)$.

\item[(ii)] If $K$ is odd-dimensional, then $K \oplus \C$ is even-dimensional and admits a Lagrangian $F \subset K \oplus \C$.
Then $L + F$ is a Lagrangian in $H_\C \oplus \C$, equivalent to the sub-Lagrangian $L \oplus \{0\}$.
Hence we obtain the completion $\Cl_{L+F}(H)$ of $\Cl_{\mathrm{alg}}(H \oplus \R) \cong \Cl_{\mathrm{alg}}(H) \otimes \Cl_1$.
In particular, we get a completion of $\Cl_{\mathrm{alg}}(H)$, as a closed subalgebra of $\Cl_{L+F}(H)$.
If $L^\prime \subset H_\C$ is a~sub-Lagrangian equivalent to $L$, \smash{$K^\prime = \big(L^\prime \oplus \overline{L}^\prime\big)^\perp$} is still odd-dimensional, by Lemma~\ref{LemmaParitySubLagrangian}, and for any Lagrangian $F^\prime \subset K^\prime \oplus \C$, $L^\prime + F^\prime$ is equivalent to $L+F$.
Hence $\Cl_{L^\prime+F^\prime}(H)$ is canonically isomorphic to $\Cl_{L + F}(H)$, and the isomorphism induces an isomorphism between the corresponding completions of $\Cl_{\mathrm{alg}}(H)$.
\end{enumerate}

We denote by $\Cl_{[L]}(H)$ the canonical von Neumann completion of $\Cl_{\mathrm{alg}}(H)$, determined by the equivalence class $[L]$ of sub-Lagrangians, as constructed above.

Observe that in the case that $K$ is even-dimensional, \eqref{IdentificationWithBH} implies that $\Cl_{[L]}(H) \cong \B(\mathcal{F}_{L^\prime})$, for any Lagrangian $L^\prime \in [L]$.
Hence $\Cl_{[L]}(H)$ is of even kind.
If $K$ is odd-dimensional, then $\Cl_{[L]}(H) \otimes \Cl_1 \cong \B(\mathcal{F}_{L})$ for some Lagrangian $L$ in $H_\C \oplus \C$ equivalent to $L \oplus \{0\}$.
Hence $\Cl_{[L]}(H)$ is of odd kind.

\subsection{The restricted orthogonal group}

The algebraic Clifford algebra $\Cl_{\mathrm{alg}}(H)$ has an action of the orthogonal group $\O(H)$ by Bogoliubov automorphisms, by its universal property.
This action extends to an action on the $C^*$-Clifford algebra $\Cl(H)$.
In contrast, by the Segal--Shale equivalence criterion, the Clifford algebra $\Cl_{[L]}(H)$ does no longer have an action of the entire orthogonal group $\O(H)$.

\begin{Definition}[restricted orthogonal group]
The \emph{restricted orthogonal group} of $H$ with respect to an equivalence class $[L]$ of sub-Lagrangians, denoted by $\O_{\res}(H, [L])$, consists of those orthogonal transformations $g$ of $H$ such that the commutator $[g, P_L]$ with the orthogonal projection $P_L$ onto $L$ is a Hilbert--Schmidt operator.
If the equivalence class $[L]$ is clear from the context, we write just $\O_{\res}(H)$.
\end{Definition}

That $\O_{\res}(H, [L])$ acts on $\Cl_{[L]}(H)$ is well known in the case that $L$ is equivalent to a Lagrangian (see, e.g., \cite[Section~6]{Araki1}).
To get the same statement in the odd case, embed
\begin{equation*}
\O_{\res}(H, [L]) \longrightarrow \O_{\res}(H \oplus \R, [L \oplus 0])
\end{equation*}
 with the upper left corner embedding.
 Then $L \oplus 0$ is equivalent to a Lagrangian and the latter group now acts on the Clifford von Neumann algebra $\Cl_{[L]}(H \oplus \R) \cong \Cl_{[L]}(H) \otimes \Cl_1$ of even kind.
 Then $\O_{\res}(H, [L])$ preserves the subalgebra
 \begin{equation*}
 \Cl_{[L]}(H) \cong \Cl_{[L]}(H) \otimes \C \subset \Cl_{[L]}(H)\otimes \Cl_1.
\end{equation*}

We always consider $\O_{\res}(H, [L])$ with the coarsest topology finer than the norm topology induced from $\O(H)$ that makes the group homomorphism
\begin{equation} \label{vonNeumannBogoliubovMap}
 \theta\colon\ \O_{\res}(H) \to \Aut(\Cl_{[L]}(H))
\end{equation}
continuous, which sends an orthogonal transformation to its Bogoliubov automorphism.
In fact, $\O_{\res}(H, [L])$ is a Banach Lie group with this topology \cite[Sections~6.2 and~2.4]{PressleySegal}.

\begin{Theorem} \label{TheoremIsoPileq2}
The map $\theta$ from \eqref{vonNeumannBogoliubovMap} induces an isomorphism on $\pi_k$ for $k \leq 5$.
\end{Theorem}

The proof uses the bundle of implementers, defined as follows.
Depending on whether the equivalence class $[L]$ contains a Lagrangian or not, we may choose a Lagrangian $L$ either in $H_\C$ or in $H_\C \oplus \C$ and let $\mathcal{F}_L$ be the corresponding Fock space.
For $g \in \O_{\res}(H, [L])$, define
\begin{equation*}
 \Imp_g = \{ U \in \U(\mathcal{F}_L) \mid \forall v \in H\colon \pi_L (g v) = U \pi_L(v) U^*\}.
\end{equation*}
By irreducibility of the Fock representation, $\Imp_g$ is a $\U(1)$-torsor.
It follows from the proof of Proposition~\ref{PropositionAutomorphismGroups} that $U$ is either even or odd.
Let $\Imp$ be the union of all $\Imp_g$, a subgroup of~$\U(\mathcal{F}_L)$.
Then $\Imp$ can be equipped with the structure of a Banach Lie group such that the map $\Imp \to \O_{\res}(H)$ is a central extension of Banach Lie groups (where the fiber over $g$ is $\Imp_g$), see \cite[Section~3.5]{KristelWaldorf1}.

\begin{proof}
First suppose that $\Cl_{[L]}(H)$ is of even kind.
In this case, we may assume that $L$ is a Lagrangian, so that $\Cl_{[L]}(H) \cong \B(\mathcal{F}_L)$.
Now, the group $\O_{\res}(H)$ is well known to have the homotopy type of the based loop space of the infinite orthogonal group \cite[Proposition~12.4.2]{PressleySegal}.
In particular, the first few homotopy groups are
\begin{equation} \label{FirstHomotopyGroupsOres}
 \pi_k(\O_{\res}(H)) = \begin{cases} \Z_2, & k=0, \\ 0, & k=1, \\ \Z, & k=2, \\ 0, & k=3, 4, 5. \end{cases}
\end{equation}
Comparing with \eqref{HomotopyGroupsAutA}, we observe that it has the same homotopy groups as $\Aut(\Cl_{[L]}(H))$ for~$k \leq 5$.
For $k\notin\{0, 2\}$, the statement of the theorem is therefore automatic, and it remains to consider the cases $k=0$ and $k=2$.

For $k=0$, we use that when $g$ does not lie in the identity component of $\O_{\res}(H)$, then $\dim\big(g \overline{L} \cap L\big)$ is odd \cite[Theorem~3.5.1]{PlymenRobinson}.
Hence the (projectively unique) unitary $U\colon \mathcal{F}_{gL} \to \mathcal{F}_L$ with $\pi_{gL}(a) = U \pi_L(a) U^*$ (which exists as $gL$ and $L$ are equivalent) is parity reversing.
Let $\Lambda_g\colon \mathcal{F}_L \to \mathcal{F}_{gL}$ be the unitary map given by taking the exterior power of $g$.
As $\Lambda_g$ is parity preserving, the unitary $U \Lambda_g$ on $\mathcal{F}_L$ is still parity reversing.
By Remark~\ref{RemarkComponentsAutA}, this implies that the $*$-automorphism of $\Aut(\Cl_{[L]}(H)) \cong \B(\mathcal{F}_L)$ given by conjugation with $U \Lambda_g$ lies in the non-identity component of $\Aut(\Cl_{[L]}(H))$. But $U \Lambda_g$ implements $\theta(g)$, hence $[\theta(g)]$ is the non-trivial element in $\pi_0(\Aut(\Cl_{[L]}(H)))$.

We now consider $k=2$.
Here we use the fact that $\Imp$ is a generator for the group of line bundles over the identity component $\O_{\res}(H)_0$, i.e., the first Chern class of $\Imp$ is a generator for $H^2(\O_{\res}(H)_0, \Z) \cong \pi_2(\O_{\res}(H)) \cong \Z$ \cite[Proposition~1.2]{SperaWurzbacherSpinors}.

On the other hand, $\Imp$ is (by definition) the pullback of the canonical line bundle over $\Aut(\Cl_{[L]}(H))$ \big(given over the identity component by $\U\big(A^0\big)$\big), the first Chern class of which is a generator for $H^2(\Aut(\Cl_{[L]}(H))_0, \Z) \cong \Z$.

But this implies that $\theta$ is an isomorphism on $H^2$, hence also on $\pi_2$ (applying the Hurewicz isomorphism to the identity component).

This finishes the proof in the even case, so we now discuss the odd case.
Then there exists a sub-Lagrangian $L$ in the fixed equivalence class such that \smash{$K = \big(L \oplus \overline{L}\big)^\perp$} is one-dimensional.
The Clifford algebra $\Cl_{[L]}(H^\prime)$, is then of even kind, where $H^\prime \subset H$ is the real subspace of $K^\perp$.
We now have the commutative diagram
\begin{equation*}
\begin{tikzcd}
 \O_{\res}(H^\prime) \ar[r] \ar[d]& \Aut(\Cl_{[L]}(H^\prime)) \ar[d]\\
 \O_{\res}(H) \ar[r]& \Aut(\Cl_{[L]}(H)),
\end{tikzcd}
\end{equation*}
where the right vertical map is $\theta \mapsto \theta \otimes \id_{\Cl_1}$ (using the isomorphism $\Cl_{[L]}(H) \cong \Cl_{[L]}(H^\prime) \otimes \Cl_1$) and the left vertical map is the upper corner embedding $\O_{\res}(H^\prime) \to \O_{\res}(H)$.
The latter is a~homotopy equivalence, as the quotient $\O_{\res}(H^\prime) / \O_{\res}(H)$ is homeomorphic to the unit sphere in~$H$, which is contractible.
The long exact sequence for homotopy groups therefore implies that the map $\O_{\res}(H^\prime) \to \O_{\res}(H)$ is a weak homotopy equivalence.
By the first part of the proof, the bottom horizontal map induces an isomorphism on $\pi_k$ for $k \leq 5$.
The same statement for the top horizontal map now follows from commutativity of the diagram and the fact that all groups involved are either $\Z$ or $\Z_2$.
\end{proof}

\section{The loop space Clifford algebra bundle}\label{sec4}

In this section, we define the loop space Clifford algebra bundle and calculate its characteristic classes.
We also discuss transgression and define the loop spin class.

\subsection{Definition of the bundle} \label{SectionLoopSpaceCliffordBundle}

Let $X$ be an oriented Riemannian manifold of dimension $d$ and let $\L X = C^\infty\big(S^1, X\big)$ be its smooth loop space.
$\L X$ is an infinite-dimensional manifold, modeled on the nuclear Fr\'echet space $C^\infty\big(S^1, \R^d\big)$.
Its tangent space $T_\gamma \L X$ at a loop $\gamma \in \L X$ can be identified with the space~$C^\infty\big(S^1, \gamma^*TX\big)$ of vector fields along $\gamma$.
It has a natural inner product coming from the standard parametrization of $S^1$ and the Riemannian metric of $X$, turning it into a pre-Hilbert space.
As we will form completed Clifford algebras (which are insensitive to whether the underlying pre-Hilbert space is complete or not \cite{PlymenRobinson}), it is natural to consider the completion
\begin{equation*}
 \mathcal{H}_\gamma \defeq L^2\big(S^1, \gamma^* TX\big),
\end{equation*}
of the tangent space.
These Hilbert spaces fit together to a bundle $\mathcal{H}$ of Hilbert spaces over~$\L X$.
To describe the bundle structure of $\mathcal{H}$, let $\SO(X)$ be the oriented frame bundle of $X$ and $\L \SO(X)$ its loop space.
$\L \SO(X)$ is a principal $\L \SO(d)$-bundle as $X$ is orientable.\footnote{Here one only needs to show that any loop $\gamma$ has a lift to $\L\SO(X)$. Such a lift exists precisely if $\gamma^*TX$ is trivializable.
Now, a vector bundle $E$ over $S^1$ is trivializable if and only if $w_1(E) = 0$. For $E = \gamma^*TX$, we have $w_1(\gamma^*TX) = \gamma^*w_1(TX)$, which is zero for all loops $\gamma$ if and only if $X$ is orientable.}
Now, $\L \SO(d)$ acts on the Hilbert space
\begin{equation}
\label{DefinitionHd}
 H^d \defeq L^2\big(S^1, \R^d\big)
\end{equation}
by pointwise multiplication, and we have the canonical identification
\begin{equation*}
 \mathcal{H} = \L\SO(X) \times_{\L\SO(d)} H^d.
\end{equation*}
Here we interpret elements $q \in \L \SO(X)$ as orthogonal transformations $H^d \to \mathcal{H}_\gamma$.
As the group~$\L\SO(d)$ acts smoothly on $H^d$, this bundle is a smooth bundle of Hilbert spaces.\footnote{We remark that often, Hilbert spaces bundles are only continuous (namely when they have the structure group $\U(H)$ with its strong topology, which is not a Lie group).
But in this case, the map $\L\SO(d) \to \O\big(H^d\big)$ is smooth when the latter group carries its Banach Lie group structure.}

For each $\gamma \in \L X$, we can form the algebraic Clifford algebra $\Cl_{\mathrm{alg}}(\mathcal{H}_\gamma)$ and its canonical $C^*$-completion $\Cl(\mathcal{H}_\gamma)$.
These algebras fit together to a continuous bundle of $C^*$-algebras, which, as above, can be identified with the associated bundle \smash{$\L \SO(X) \times_{\L\SO(d)} \Cl\big(H^d\big)$}, where $\L\SO(d)$ acts on $\Cl\big(H^d\big)$ through $\O\big(H^d\big)$, by Bogoliubov automorphisms.
Here it is important that the homomorphism $\L\SO(d) \to \O\big(H^d\big) \to \Aut\big(\Cl\big(H^d\big)\big)$ is continuous when $\O\big(H^d\big)$ carries the norm topology \cite[Proposition~4.35]{Ambler2012}; this implies that $\Cl(\mathcal{H})$ has the structure of a continuous bundle of $C^*$-algebras.
However, this bundle seems rather uninteresting for loop space spin geometry, by the following.

\begin{Theorem}
 The bundle $\Cl(\mathcal{H})$ is trivial.
\end{Theorem}

\begin{proof}
The action of $\L\SO(d)$ extends to an action of the orthogonal group $\O\big(H^d\big)$ of $H^d$, which is contractible by Kuiper's theorem.
So $\Cl(\mathcal{H})$ is an associated bundle for the principal $\O\big(H^d\big)$-bundle $\L\SO(X) \times_{\L\SO(d)} \O\big(H^d\big)$, which must be trivial by contractibility of $\O\big(H^d\big)$ and its classifying space $B\O\big(H^d\big)$.
Hence $\Cl(\mathcal{H})$ is trivial as well.
\end{proof}

\begin{Remark}
As an infinite tensor product algebra, $\Cl\big(H^d\big)$ is an example of a so-called \emph{strongly self-absorbing} $C^*$-algebra, which has a contractible automorphism group \cite{DardalatPennig,DardalatWinter}.
Hence \emph{any} bundle with typical fiber $\Cl\big(H^d\big)$ must in fact be contractible.
However, this argument does not take into account the grading or the real structure, while the above argument also shows that~$\Cl(\mathcal{H})$ is trivial as a bundle of graded, real $C^*$-algebras.
\end{Remark}

In order to obtain a non-trivial bundle, we now construct a suitable completion of $\Cl_{\mathrm{alg}}(\mathcal{H})$ to a bundle of von Neumann algebras.
To this end, we observe that the model space $H^d_\C$ admits a canonical sub-Lagrangian
\begin{equation}
\label{SubLagrangianLd}
 L^d = \mathrm{span}\big\{ \xi \otimes{\rm e}^{\rm int} \mid n \in \N, \, \xi \in \C^d\big\}.
\end{equation}
The space $\big(L^d + \overline{L}^d\big)^\perp$ is just the space of constant functions on the circle, which has dimension~$d$.
By the discussion in Section~\ref{SectionCliffordAlgebras}, we therefore obtain a canonical von Neumann completion
\begin{equation} \label{DefinitionAd}
A_d \defeq \Cl_{[L^d]}\big(H^d\big),
\end{equation}
which is of even kind when $d$ is even and of odd kind when $d$ is odd.

To obtain a von Neumann completion $\A_\gamma$ of $\Cl_{\mathrm{alg}}(\mathcal{H}_\gamma)$ for $\gamma \in \L X$, we observe that any lift~${q \in \L \SO(X)}$ gives a Lagrangian $q L^d \subset \mathcal{H}_\gamma$.
It is now crucial that the multiplication action of~$\L\SO(d)$ on $H^d$ is in fact by elements of the restricted orthogonal group $\O_{\res}\big(H^d\big) = \O_{\res}\big(H^d, \big[L^d\big]\big)$ and that we get a continuous group homomorphism
\begin{equation} \label{LSOtoOres}
 \L\SO(d) \longrightarrow \O_{\res}\big(H^d\big),
\end{equation}
see \cite[Proposition~6.3.1]{PressleySegal}, \cite[Proposition~3.23]{KristelWaldorf1}.
Hence for any two lifts $q, q^\prime \in \L \SO(X)$ of $\gamma$, the Lagrangians $q L^d$ and $q^\prime L^d$ are equivalent.
We therefore obtain a well-defined von Neumann completion $\A_\gamma \defeq \Cl_{[q L^d]}(\mathcal{H}_\gamma)$, independent of the choice of $q$.
These algebras can be canonically identified with the fibers of the bundle
\begin{equation} \label{vNBundleAsAssociatedBundle}
 \A_{\L X} \defeq \L \SO(X) \times_{\L\SO(d)} A_d.
\end{equation}
This is a continuous bundle of von Neumann algebras as the homomorphism
\begin{equation} \label{ThetaTilde}
\begin{tikzcd}
\tilde{\theta}\colon \ \L \SO(d) \ar[r] & \O_{\res}\big(H^d\big) \ar[r] & \Aut(A_d)
\end{tikzcd}
\end{equation}
obtained by composing \eqref{LSOtoOres} with the Bogoliubov action \eqref{vonNeumannBogoliubovMap} is continuous.

\begin{Remark}
An alternative construction of the von Neumann completion of $\Cl_{\mathrm{alg}}(\mathcal{H}_\gamma)$ is the following.
Let $D_\gamma = \mathrm{i}\frac{\nabla}{{\rm d}t}$ be the operator acting on the bundle $\gamma^*T X \otimes \C$ using the pullback of the Levi-Civita connection on $TX$.
Let $L_\gamma$ be the Hilbert space direct sum of eigenspaces to negative eigenvalues of $D_\gamma$.
This is a sub-Lagrangian in $\mathcal{H}_\gamma^\C$, which can be shown to be equivalent to $q L^d$ for any $q \in \L\SO(d)$.
For details, see \cite[Section~1.3]{LudewigSpinorBundle}.
\end{Remark}

\begin{Remark}
A closely connected bundle of Clifford algebras on $\L X$ has recently been considered in the somewhat different context of \emph{rigged} von Neumann algebra bundles by Kristel and Waldorf \cite{KristelWaldorf2}.
\end{Remark}

For a Lie group $G$, we denote by $\Omega G \subset \L G$ the based loop space of $G$, i.e., the set of smooth loops $\gamma\colon S^1 \to G$ with $\gamma(0) = e$, the neutral element of $G$.
For $G = \SO(d)$, the restriction of~\eqref{LSOtoOres} to based loops gives a continuous group homomorphism $\Omega \SO(d) \longrightarrow \O_{\res}\big(H^d\big)$.

\begin{Lemma} \label{LemmaPi2Iso}
If $d \geq 5$, the above homomorphism induces an isomorphism on $\pi_k$ for $k \leq 2$.
\end{Lemma}

\begin{proof}
For $d$ even, this statement is well known: By \cite[Proposition~12.5.2]{PressleySegal}, for any $m \in \N$, the map $\Omega\SO(2m) \to \O_{\res}(H^{2m})$ is $(2m-3)$-connected.
This shows the claim in even dimensions $d=2m \geq 6$.
For $d \geq 5$ odd, we consider the commutative diagram
\begin{equation} \label{DiagramOddEven}
\begin{tikzcd}
\Omega\SO(d) \ar[d] \ar[r] & \O_{\res}\big(H^d\big) \ar[d] \\
\Omega\SO(d+1) \ar[r] & \O_{\res}\big(H^{d+1}\big),
\end{tikzcd}
\end{equation}
the bottom map of which is an isomorphism on $\pi_k$ for $k \leq d-3$ by the discussion for the even case.
The right vertical map in \eqref{DiagramOddEven} is the restriction of the map
\[\O\big(H^d\big) \to \O\big(H^d\big) \times \O\big(H^1\big) \subset \O\big(H^{d+1}\big).\]
 The left vertical map of the diagram is induced by the canonical embedding $\SO(d) \to \SO(d+1)$, which induces an isomorphism on $\pi_k$ for $k \leq 3$. By \cite[Proposition~7.1]{CareyCrowleyMurray}, the inclusion of~$\Omega \SO(d)$ into the \emph{continuous} based loop space of $\SO(d)$ is a homotopy equivalence. Hence $\Omega \SO(d) \to \Omega\SO(d+1)$ induces an isomorphism on $\pi_k$ for $k \leq 2$.

By the above considerations and the commutativity of \eqref{DiagramOddEven}, the map \[\pi_k(\Omega\SO(d)) \to \pi_k\big(\O_{\res}\big(H^d\big)\big)\]
 is injective, and the right vertical map in \eqref{DiagramOddEven} is surjective on $\pi_k$, for $k \leq 2$.
As seen in \eqref{FirstHomotopyGroupsOres}, the first few homotopy groups of $\O_{\res}\big(H^d\big)$ are $\Z_2$, $0$ and $\Z$, which implies that all maps in the diagram must be isomorphisms on $\pi_k$.
\end{proof}

Combining the above result with Theorem~\ref{TheoremIsoPileq2}, we obtain the following result.

\begin{Corollary} \label{CorollaryDoubleIso}
 If $d \geq 5$, the composition
 \begin{equation*}
 \begin{tikzcd}
 \tilde{\theta}\colon\ \Omega \SO(d) \ar[r] & \O_{\res}\big(H^d\big) \ar[r] & \Aut(A_d)
 \end{tikzcd}
 \end{equation*}
 induces an isomorphism on $\pi_k$ for $k \leq 2$.
\end{Corollary}

\subsection{Transgression and the loop spin class}
\label{SectionTransgression}

For a manifold $Y$ and a coefficient group $R$, \emph{transgression} is the composition
\begin{equation*}
\begin{tikzcd}
 \tau\colon\ H^k(Y, R) \ar[r, "\mathrm{ev}^*"] & H^k\big(\L Y \times S^1, R\big) \ar[r, "\int_{S^1}"] & H^{k-1}(\L Y, R),
\end{tikzcd}
\end{equation*}
where the left map is pullback with the evaluation map $\mathrm{ev}\colon \L Y \times S^1 \to Y$ and the right map is fiber integration over the $S^1$ factor.
Transgression is natural, in the sense that for a smooth map $f\colon Y \to Y^\prime$, the diagram
\begin{equation*}
 \begin{tikzcd}
 H^k(Y^\prime, R) \ar[d, "f^*"'] \ar[r, "\tau"] & H^{k-1}(\L Y^\prime, R) \ar[d, "\L f^*"] \\
H^k(Y, R) \ar[r, "\tau"] & H^{k-1}(\L Y, R)
 \end{tikzcd}
\end{equation*}
commutes. Let $G = \SO(d)$ or $\Spin(d)$.
The classifying spaces $BG$ and the universal bundle $EG$ admit an infinite dimensional manifold model, so that their smooth loop space $\L BG$ is well-defined.
As $\L EG$ is again contractible and $\L G$ acts freely on it, the quotient $\L BG$ is a model for the classifying space $B \L G$.
We therefore have transgression homomorphisms
\begin{equation*}
 \tau\colon\ H^k(BG, R) \longrightarrow H^{k-1}(B\L G, R).
\end{equation*}

\begin{Remark}
Of course, transgression is also defined for general topological spaces, using the continuous loop space instead of the smooth version.
However, as we work with the smooth loop space throughout, we presented the construction in this case.
We recall that the smooth loop space of a manifold is homotopy equivalent to the continuous loop space, and the same is true for based loop spaces \cite[Proposition~7.1]{CareyCrowleyMurray}.
\end{Remark}

The base loop group $\Omega G$ is the kernel of the evaluation-at-zero map $\mathrm{ev}_0\colon \L G \to G$, so we have the short exact sequence
\begin{equation} \label{PathSequence}
 \begin{tikzcd}
 \Omega G \ar[r] & \L G \ar[r] & G \ar[l, bend right=40, "\iota"'],
\end{tikzcd}
\end{equation}
which is split via the inclusion $\iota\colon G \to \L G$ as constant loops.
This induces a fibration of the corresponding classifying spaces, and an exact sequence
\begin{equation} \label{SESRemark}
\begin{tikzcd}
H^3(B\SO(d), \Z) \ar[r] & H^3(B\L\SO(d), \Z) \ar[r] & H^3(B\Omega\SO(d), \Z)
\end{tikzcd}
\end{equation}
on the corresponding cohomology groups.
The first group here is $\Z_2$, generated by the third universal integral Stiefel--Whitney class $W_3$ (i.e., the Bockstein image of the second universal Stiefel--Whitney class $w_2$).
As $B\Omega\SO(d) \simeq \SO(d)$, the right group equals $H^3(\SO(d), \Z) \cong \Z$.

\begin{Lemma} \label{LemmaH3BLSOZ}
If $d \geq 5$, the sequence \eqref{SESRemark} is split exact, hence we have a canonical isomorphism
\begin{equation*}
 H^3(B\L\SO(d), \Z) \cong \Z \times \Z_2.
\end{equation*}
\end{Lemma}

\begin{proof}
As \eqref{PathSequence} is split exact, with $G = \SO(d)$ including into $\L G$ as constant loops, the corresponding fibration of classifying spaces admits a section $B\iota$.
This induces a left split of~\eqref{SESRemark}.
That the sequence \eqref{SESRemark} is exact in the middle follows from the Serre spectral sequence for the classifying space fibration of \eqref{PathSequence}.
The right map in \eqref{SESRemark} is surjective by the following argument.
By Corollary~\ref{CorollaryDoubleIso}, the map $B\tilde{\theta}\colon B\Omega \SO(d) \to B \Aut(A_d)$ induces an isomorphism on $\pi_k$ for~$k \leq 3$.
It is moreover trivially surjective on $\pi_4$ as $\pi_4(B\Aut(A_d)) = 0$.
Hence pullback induces an isomorphism
\begin{equation*}
 B\tilde{\theta}^* \colon\ H^3(B \Aut(A_d, \Z)) \stackrel{\cong}{\longrightarrow} H^3(B\Omega \SO(d), \Z).
\end{equation*}
On the other hand, the group homomorphism $\tilde{\theta} \colon \Omega \SO(d) \to \Aut(A_d)$ extends to all of $\L \SO(d)$, hence the above isomorphism factors through $H^3(B\L \SO(d), \Z)$.
This shows that the right map in \eqref{SESRemark} must be surjective.
\end{proof}

We are now interested in the following commutative diagram
\begin{equation}\label{TransgressionDiagramTheorem}
\begin{tikzcd}
 H^3(B\L\SO(d), \Z) \ar[r] & H^3(B\L\Spin(d), \Z) \\
 H^4(B\SO(d), \Z) \ar[u, "\tau"] \ar[r, "\cdot 2"] & H^4(B\Spin(d), \Z), \ar[u, "\tau", "\cong"']
\end{tikzcd}
\end{equation}
where $d\geq 5$.
The bottom left group is $\Z$, generated by the universal Pontrjagin class $p_1$.
The bottom right group is also $\Z$, generated by the fractional Pontrjagin class $\frac{1}{2} p_1$, and the bottom horizontal map has been shown to be multiplication by two on generators by McLaughlin \cite[proof of Lemma~2.2]{McLaughlin}, more specifically it sends $p_1$ to the class $2 \cdot \frac{1}{2} p_1$.
The right transgression map is an isomorphism as $B\Spin(d)$ is 2-connected \cite[p.~149]{McLaughlin}.

\begin{Proposition} \label{PropCommDiagramH3Transgression}
If $d \geq 5$, the top horizontal map in \eqref{TransgressionDiagramTheorem} is surjective.
\end{Proposition}

For the proof, we need the following lemma.

\begin{Lemma} \label{LemmaH3IsoSoSpin}
The map $H^3(B\Omega\SO(d), \Z) \to H^3(B\Omega\Spin(d), \Z)$ is an isomorphism.
\end{Lemma}

\begin{proof}
Observe that $\Omega\Spin(d)$ is canonically identified with the identity component of $\Omega\SO(d)$.
We therefore obtain the commutative diagram
\begin{equation*}
\begin{tikzcd}
 \Omega \Spin(d) = \Omega \SO(d)_0 \ar[r] \ar[d] & \Aut(A_d)_0 \ar[d] \\
 \Omega \SO(d) \ar[r] & \Aut(A_d).
\end{tikzcd}
\end{equation*}
The horizontal maps of this diagram induce isomorphisms on $\pi_k$ for $k\leq 2$, by Corollary~\ref{CorollaryDoubleIso}.
It follows that in the corresponding diagram on classifying spaces, the horizontal maps induce isomorphisms on $\pi_k$ for $k \leq 3$.
They are moreover trivially surjective on $\pi_4$ (since this group is trivial for the right-hand side).
By Whitehead's theorem \cite[Theorem~10.28]{Switzer}, these maps also induce isomorphisms on $H_k$, for $k \leq 3$ (and a surjection on $H_4$).
From the universal coefficient theorem and the five lemma, we obtain that they also induce isomorphisms on $H^k$, for $k \leq 3$ (and a surjection on $H^4$).
Moreover, the map $H^3(B\Aut(A_d), \Z) \to H^3(B\Aut(A_d)_0, \Z)$ is an isomorphism by Corollary~\ref{CorollaryHomotopyType}.
The lemma follows.
\end{proof}

\begin{proof}[Proof of Proposition~\ref{PropCommDiagramH3Transgression}]
Consider the commutative diagram
\begin{equation*}
\begin{tikzcd}
 H^3(B\L\SO(d), \Z) \ar[d, two heads] \ar[r] & H^{3}(B\L\Spin(d), \Z) \ar[d]\\
 H^3(B\Omega\SO(d), \Z) \ar[r, "\cong"] & H^3(B\Omega\Spin(d), \Z).
 \end{tikzcd}
\end{equation*}
The bottom horizontal map is an isomorphism by Lemma~\ref{LemmaH3IsoSoSpin}.
The left horizontal map is surjective by Lemma~\ref{LemmaH3BLSOZ}.
Now the counterclockwise composition is surjective, hence so must be the clockwise composition.
As the two rightmost groups are isomorphic to $\Z$, we obtain in particular that the top horizontal map must be surjective, as claimed.
\end{proof}

We conclude that the square \eqref{TransgressionDiagramTheorem} takes the following form:
\begin{equation}\label{TransgressionDiagram2}
\begin{tikzcd}
 \Z \times \Z_2 \ar[r, "\mathrm{pr}_1"] & \Z \ar[d, equal] \\
 \Z \ar[u, "(\cdot 2\text{,} 0)"] \ar[r, "\cdot 2"] &\, \Z.
\end{tikzcd}
\end{equation}

\begin{Corollary} \label{CorollaryStringClass}
If $d \geq 5$, there is a unique class $\mathfrak{S} \in H^3(B\L\SO(d), \Z)$ such that
\begin{equation*}
2 \cdot \mathfrak{S} = \tau(p_1) \qquad \text{and} \qquad B\iota^*\mathfrak{S} = 0,
\end{equation*}
where $B\iota\colon B\SO(d) \to B\L\SO(d)$ is induced by the splitting \eqref{PathSequence}.
Moreover, under the homomorphism $H^3(B\L\SO(d), \Z) \to H^3(B\L\Spin(d), \Z)$, this class is sent to $\tau\big(\frac{1}{2}p_1\big)$.
\end{Corollary}

\begin{proof}
Using Lemma~\ref{LemmaH3BLSOZ}, an inspection of the diagram \eqref{TransgressionDiagram2} shows that there are two solutions for the equation $2 \cdot \mathfrak{S} = \tau(p_1)$, which differ by the 2-torsion class $\mathrm{ev}_0^*W_3$.
The condition $B\iota^*\mathfrak{S} = 0$ removes this ambiguity.
The additional statement also follows from the commutative diagram~\eqref{TransgressionDiagram2}.
\end{proof}

\begin{Definition}[universal loop spin class] \label{DefinitionStringClass}
We call the class $\mathfrak{S} \in H^3(B\L\SO(d), \Z)$ from Corollary~\ref{CorollaryStringClass} the \emph{universal loop spin class}.
\end{Definition}

\subsection{Proof of the main theorem}

In this section, we express the characteristic classes of $\A_{\L X}$ in terms of transgression and hence prove the main theorem from the introduction.
Let $X$ be an oriented Riemannian manifold of dimension $d\geq5$ and let $f\colon X \to B\SO(d)$ be the classifying map for its oriented frame bundle $\SO(X)$.
Then the looped map $\L f\colon \L X \to \L B\SO(d) = B\L\SO(d)$ classifies the principal $\L\SO(d)$-bundle $\L\SO(X)$.

\begin{Definition}[loop spin class]
The \emph{loop spin class of} $X$ is
\begin{equation*}
\mathfrak{S}(X) \defeq \L f^* \mathfrak{S} \in H^3(\L X, \Z).
\end{equation*}
\end{Definition}

Let $\A_{\L X}$ be the Clifford von Neumann algebra bundle on $\L X$ constructed in Section~\ref{SectionLoopSpaceCliffordBundle}.
Its typical fiber is the Clifford von Neumann algebra $A_d$ from \eqref{DefinitionAd}.
By Proposition~\ref{PropositionClassifyingWStar}, the bundle is therefore classified by a map
\begin{equation*}
h\colon\ \L X \longrightarrow B\Aut(A_d).
\end{equation*}
The characteristic classes of $\A_{\L X}$ are then by definition the pullback along $h$ of the universal classes $\oclass$ and $\DD$ on $B\Aut(A_d)$. On the other hand, by \eqref{vNBundleAsAssociatedBundle}, $\A_{\L X}$ is an associated bundle to~$\L\SO(X)$, via the action \eqref{ThetaTilde} of $\L\SO(d)$ on $A_d$.
This means that the classifying map $h$ of the bundle $\A_{\L X}$ admits the factorization
\begin{equation*}
\begin{tikzcd}
 h \colon\ \L X \ar[r, "\L f"] 
 & B \L \SO(d) \ar[r, "B\tilde{\theta}"] & B\Aut(A_d).
\end{tikzcd}
\end{equation*}

\begin{proof}[Proof of Main Theorem~\ref{ThmObstructionClass}]
We first consider the orientation class.
To this end, we consider the following commutative diagram:
\begin{equation} \label{H1Diagram}
\begin{tikzcd}
 & & H^{1}(B\Aut(A_d), \Z_2) \ar[dl, "B\tilde{\theta}^*"'] \ar[d, "\cong"] \ar[dll, "h^*"', bend right=10]\\
 H^{1}(\L X, \Z_2) & \ar[l, "\L f^*"] H^{1}(B\L \SO(d), \Z_2) \ar[r] & H^{1}(B\Omega\SO(d), \Z_2) \\
 H^2(X, \Z_2) \ar[u, "\tau"] & \ar[l, "f^*"'] H^2(B\SO(d), \Z_2) \ar[u, "\tau", "\cong"'].
\end{tikzcd}
\end{equation}
Here all groups independent of $X$ are $\Z_2$, and all the maps independent of $X$ are isomorphisms:
Indeed, it is well known that $H^2(B\SO(d), \Z_2) = \Z_2$, generated by the universal Stiefel--Whitney class $w_2$.
That the right transgression map is an isomorphism has been shown by McLaughlin \cite[proof of Proposition~2.1]{McLaughlin}. This implies $H^{1}(B\L \SO(d), \Z_2) \cong \Z_2$.
That also ${H^1(B\Aut(A_d), \Z_2) = \Z_2}$ follows from \eqref{HomotopyGroupsAutA} and the Hurewicz isomorphism.
The rightmost vertical map is an isomorphism by Corollary~\ref{CorollaryDoubleIso}.
It follows that the map $H^{1}(B\L \SO(d), \Z_2) \to H^{1}(B\Omega\SO(d), \Z_2)$ is surjective.
That it is also injective is clear as all groups involved are $\Z_2$.

The above discussion implies that $B\tilde{\theta}^* \oclass = \tau(w_2)$.
Hence using commutativity of the left rectangle in \eqref{H1Diagram}, we obtain
\begin{equation*}
 h^* \oclass = \L f^*B\tilde{\theta}^*\oclass = \L f^* \tau(w_2) = \tau(f^*w_2) = \tau(w_2(X)),
\end{equation*}
where we recall that the pullback $f^*w_2$ is by definition the second Stiefel--Whitney class $w_2(X)$ of $X$.
This proves the claim.

For the third integer cohomology, we consider the commutative diagram
\begin{equation} \label{H3Diagram}
\begin{tikzcd}
 & & H^{3}(B\Aut(A_d), \Z) \ar[dl, "B\tilde{\theta}^*"'] \ar[dll, "h^*"', bend right = 10] \ar[d, "\cong"] \\
 H^{3}(\L X, \Z) & \ar[l, "\L f^*"] H^{3}(B\L \SO(d), \Z) \ar[r, two heads] & H^{3}(\B\L\Spin(d), \Z) \\
 H^4(X, \Z) \ar[u, "\tau"] & \ar[l, "f^*"'] H^4(B\SO(d), \Z) \ar[u, "\tau"] \ar[r, "\cdot 2"] & H^4(B \Spin(d), \Z). \ar[u, "\tau"', "\cong"]
\end{tikzcd}
\end{equation}
Here we use that $\L\Spin(d)$ acts on $A_d$ along the homomorphism $\L p\colon \L\Spin(d) \to \L\SO(d)$, which induces a map $B\L\Spin(d) \to B\Aut(A_d)$.
The bottom right square is just \eqref{TransgressionDiagramTheorem}.

\begin{Lemma} \label{LemmaIsomorphismTopRightH3}
The top right vertical map in \eqref{H3Diagram} is an isomorphism.
\end{Lemma}

\begin{proof}
Consider the following commutative diagram:
\begin{equation} \label{TentDiagram}
\begin{tikzcd}
H^3(B \Aut(A_d), \Z) \ar[d, "\cong"'] \ar[r] \ar[dr]
&
H^3(B\L\Spin(d), \Z)
\ar[d]
 \\
H^3(B\Omega\SO(d), \Z) \ar[r, "\cong"'] & H^3(B\Omega\Spin(d), \Z).
\end{tikzcd}
\end{equation}
The left vertical map is an isomorphism by Corollary~\ref{CorollaryDoubleIso}.
The bottom horizontal map is an isomorphism by Lemma~\ref{LemmaH3IsoSoSpin}.
Hence the diagonal map in \eqref{TentDiagram} is an isomorphism.
The fact that all groups involved are isomorphic to $\Z$ then implies that also the right vertical map in~\eqref{TentDiagram} is an isomorphism.
\end{proof}

Both the top and bottom right groups in \eqref{H3Diagram} are canonically isomorphic to $\Z$, with generators the universal Dixmier--Douady class $\DD$, respectively the fractional universal Pontrjagin class $\frac{1}{2}p_1$.
By Lemma~\ref{LemmaIsomorphismTopRightH3} and the fact that the bottom right map is an isomorphism (see the discussion below \eqref{TransgressionDiagramTheorem}), we observe that the transgression of $\frac{1}{2}p_1$ equals the pullback of $\DD$ along the top right vertical map in \eqref{H3Diagram}, up to a possible sign.
This sign turns out to be $+1$, but establishing this involves rather intricate calculations which we defer to the appendix.
By a diagram chase, we then get that $2\cdot B\tilde{\theta}^* \DD = \tau(p_1)$, hence
\begin{equation*}
 \tau(p_1(X)) = \tau(f^* p_1) = \L f^* \tau(p_1) = 2\cdot \L f^* B\tilde{\theta}^* \DD = 2 \cdot h^* \DD = 2\cdot\DD(\A_{\L X}).
\end{equation*}
This finishes the proof of Main Theorem~\ref{ThmObstructionClass}.
\end{proof}

\begin{Remark} \label{RemarkSign}
Replacing the Lagrangian $L^d$ with $\big(L^d\big)^\perp$ results in a different von Neumann algebra bundle $\tilde{\mathcal{A}}_{\L X}$, satisfying $\DD\big(\tilde{\mathcal{A}}_{\L X}\big) = -\DD(\mathcal{A}_{\L X})$.
This follows from the fact that the Lie algebra cocycle for the group extension $\Imp \to \O_{\res}\big(H^d\big)$ is replaced by its negative under this change, by the calculation in \cite[Theorem~6.10]{Araki1}.
Hence any choice of sign for the Dixmier--Douady class (in comparison to $\tau(p_1)$) can be achieved by a modification of the Clifford algebra construction.
\end{Remark}

The following is a more refined version of Main Theorem~\ref{ThmObstructionClass}.

\begin{Theorem} \label{ThmRepeatedMain}
Let $d\geq 5$. Then we have
\begin{equation*}
\oclass(\A_{\L X}) = \tau(w_2), \qquad \DD(\A_{\L X}) = \mathfrak{S}(X) + \mathrm{ev}_0^*W_3(X).
\end{equation*}
\end{Theorem}

\begin{proof}
With a view on Lemma~\ref{LemmaH3BLSOZ}, it follows from Main Theorem~\ref{ThmObstructionClass} that
\begin{equation*}
\DD(\mathcal{A}_{\L X}) = \mathfrak{S}(X) + n \, \mathrm{ev}_0^*W_3(X), \qquad \text{for}~~n \in \{0, 1\}.
\end{equation*}
We have $\iota_X^* \mathfrak{S}(X) = 0$ and $\iota_X^*\mathrm{ev}_0^*W_3(X) = W_3(X)$, where $\iota_X \colon X \to \L X$ is the inclusion as constant loops and $\mathrm{ev}_0 \colon \L X \to X$ is evaluation at zero.
For the proof of Theorem~\ref{ThmRepeatedMain}, it is therefore left to show that $\iota_X^* \DD(\mathcal{A}_{\L X}) = W_3(X)$.
By naturality of the Dixmier--Douady class, we have $\iota_X^* \DD(\mathcal{A}_{\L X}) = \DD(\iota_X^*\mathcal{A}_{\L X})$.

Now, the bundle $\iota_X^*\mathcal{A}_{\L X}$ over $X$ can be written as an associated bundle,
\begin{equation*}
 \iota_X^*\mathcal{A}_{\L X} \cong \SO(X) \times_{\SO(d)} A_d,
\end{equation*}
where $\SO(d)$ acts on $A_d$ through Bogoliubov automorphisms, induced by the multiplication with constant loops on $H^d$.
Denote by $K\subset H^d$ the subspace of constant functions.
Then both $K_\C$ and~$L^d$ are invariant under the action of $\SO(d)$ and $L^d$ is a Lagrangian in $H^\prime = K_\C^\perp$.
Let $A_d^\prime = \B(\mathcal{F}_{L^d})$ be the von Neumann algebra completion of $\Cl_{\mathrm{alg}}(H^\prime)$ with respect to this Lagrangian.
Then identifying $\Cl(K) \cong \Cl_d$, we have $A_d = A^\prime \,\bar{\otimes}\, \Cl_d$ and $\iota_X^*\mathcal{A}_{\L X}$ splits as a~tensor product, $\iota_X^*\mathcal{A}_{\L X} \cong \mathcal{A}^\prime \bar{\otimes}\Cl(X)$, where $\mathcal{A}^\prime = \SO(X) \times_{\SO(d)} A^\prime$ and $\Cl(X)$ is the usual complex Clifford algebra bundle on $X$.
By Theorem~\ref{PropTensorProductDD}, we have
\begin{equation*}
 \DD(\iota_X^*\mathcal{A}_{\L X}) = \DD(\mathcal{A}^\prime) + \DD(\Cl(X)) + \beta(\oclass(\mathcal{A}^\prime) \smile \oclass(\Cl(X))).
\end{equation*}
It is well known that $\DD(\Cl(X)) = W_3(X)$ (see \cite[Lemma~7]{DonovanKaroubi} and \cite{MathaiMelroseSingerProjectiveFamilies}).

We show that $\mathcal{A}^\prime$ is trivializable, which finishes the proof.
To this end, observe that as $\mathcal{A}^\prime$ is associated to $\SO(X)$, its classifying map $X \to B\Aut(A^\prime)$ factors through $B\SO(d)$.
We now show that the map $B\SO(d) \to B\Aut(A^\prime)$ is contractible.
To this end, observe that the action of $\iota(\SO(d)) \subset \L \SO(d)$ preserves $L^d$.
This means that for each $q \in \SO(d)$, multiplication by $\iota(q)$ commutes with the complex structure \smash{$J = i\big(P_{L^d}-P_{\overline{L}^d}\big)$}.
The image of $\iota(\SO(d))$ in $\O_{\res}\big(H^d\big)$ therefore lies in the subgroup $\U\big(H^d_J\big) \subset \O_{\res}\big(H^d\big)$, which is contractible.
Hence the classifying map $\SO(d) \to \L\SO(d) \to \O_{\res}\big(H^d\big) \to \Aut(A^\prime)$ of $\mathcal{A}^\prime$ is null-homotopic, and so is the induced map on classifying spaces.
This proves the claim.
\end{proof}

\subsection{A twisted Clifford algebra bundle} \label{SectionTwistedBundle}

There is a variant for the construction of the loop space Clifford algebra bundle which takes the model space $
 H^d_{\mathbb{S}} = L^2\big(S^1, \R^d \otimes \mathbb{S}\big)$ as input, where $\mathbb{S}$ is the M\"obius bundle over $S^1$.
This has been considered, e.g., in \cite[Section~6.2]{KristelWaldorf2}.
The main difference here is that this space has a~canonical Lagrangian $L_{\mathbb{S}}^d$, which under the identification of $H^d_{\mathbb{S}}$ with $2\pi$-anti-periodic functions on $\R$ can be written as
\begin{equation}
\label{OtherLagrangian}
 L^d_{\mathbb{S}} = \big\{ \xi \otimes {\rm e}^{{\rm i}t(n+\frac{1}{2})} \mid n \in \N_0, \, \xi \in \R^d\big\}.
\end{equation}
Hence the corresponding Clifford von Neumann algebra $A_d^{\mathbb{S}} = \Cl_{[L^d_{\mathbb{S}}]}\big(H^d_{\mathbb{S}}\big)$ is of even kind in any~di\-men\-sion $d$.
In a similar fashion to before, we obtain a bundle $\A^{\mathbb{S}}_{\L X}$ of super type I factors.

\begin{Theorem}
Let $d \geq 5$.
Then the characteristic classes of $\A_{\L X}^\mathbb{S}$ are
\begin{equation*}
 \oclass\big(\mathcal{A}_{\L X}^{\mathbb{S}}\big) = \tau(w_2(X)), \qquad \DD\big(\A_{\L X}^{\mathbb{S}}\big) = \mathfrak{S}(X).
\end{equation*}
\end{Theorem}

\begin{proof}
This is shown analogously to the proof of Theorem~\ref{ThmRepeatedMain}.
The only difference is the calculation of $\iota_X^* \DD\big(\mathcal{A}_{\L X}\big) = \DD\big(\iota_X^*\mathcal{A}_{\L X}^{\mathbb{S}}\big)$.
In this case, $\iota_X^*\mathcal{A}_{\L X}^{\mathbb{S}}$ can be identified with the associated bundle \smash{$\SO(X) \times_{\SO(d)} A_d^{\mathbb{S}}$}.
The point is now that the action of $\SO(d)$ preserves the Lagrangian~$L^d_{\mathbb{S}}$, hence the homomorphism $\SO(d) \to \O_{\res}\big(H^d_{\mathbb{S}}\big) \to \Aut\big(A_d^{\mathbb{S}}\big)$ factors through the contractible subgroup \smash{$\U\big(\big(H^d_{\mathbb{S}}\big)_J\big)$}, with $J$ the complex structure determined by $L^d_{\mathbb{S}}$.
Therefore, ${\iota_X^*\mathcal{A}_{\L X}^d}$ is trivializable and $\DD\big(\iota_X^*\mathcal{A}_{\L X}^{\mathbb{S}}\big)=0$.
\end{proof}

\appendix

\section{Sign discussion}
\label{AppendixSign}

In this appendix, we fix the sign indeterminacy present in the proof of Theorem~\ref{ThmRepeatedMain} above.
Precisely, we prove the following.

\begin{Proposition}
\label{SignProposition}
For $d\geq 5$, the transgression
\begin{equation*}
\tau(p_1) \in H^3(\L B\SO(d), \Z) = H^3(B\L\SO(d), \Z)
\end{equation*}
of the universal first Pontrjagin class equals two times the pullback of the universal Dixmier--Douady class $\DD \in H^3(B\Aut(A_d), \Z)$ along the map on classifying spaces induced by the composition
\begin{equation*}
\begin{tikzcd}
\tilde{\theta} \colon\ \L\SO(d) \ar[r, "j"] & \O_{\mathrm{res}}\big(H^d\big) \ar[r, "\theta"] & \Aut(A_d).
\end{tikzcd}
\end{equation*}
\end{Proposition}

For definiteness, we emphasize that we use the (standard) convention for Pontrjagin classes that for any complex line bundle $L$ with underlying real bundle $L_\R$, we have $p_1(L_\R) = c_1(L)^2$.

We recall the definition of the universal Dixmier--Douady class $\DD$.
To begin with, recall that for a topological group $G$ with classifying space fibration $G \to EG \to BG$, there is a~homomorphism
\begin{equation*}
 \tilde{\tau} \colon\ H^k(BG, R) \longrightarrow H^{k-1}(G, R),
\end{equation*}
natural in $G$, which is called transgression and should not be confused with the notion of transgression discussed in Section~\ref{SectionTransgression}.\footnote{Both this notion of transgression and the one discussed in Section~\ref{SectionTransgression} are special cases of the more general notion of transgression in a fiber bundle \cite{Borel}.}
In the case of $G = \Aut(A_d)$ and $k=3$, the transgression homomorphism
\begin{equation}
\label{TildeTransgression}
\tilde{\tau} \colon\ H^3(B\Aut(A_d), \Z) \longrightarrow H^2(\Aut(A_d), \Z)	
\end{equation}
is an isomorphism, and the universal Dixmier--Douady class $\DD$ is by definition the class that is sent to the first Chern class $c_1(G_d) \in H^2(\Aut(A_d), \Z)$ of the canonical $\U(1)$-bundle $G_d$ over~$\Aut(A_d)$.

As the classes we are interested in are not torsion, we may work over real coefficients.
We consider the diagram
\begin{equation*}
\begin{tikzcd}[column sep=0.8cm]
H^3(B\Aut(A_d), \R)
	\ar[r, "B\theta^*", "\cong"']
	\ar[d, "\tilde{\tau}"]
&
H^3\big(B\O_{\res}\big(H^d\big), \R\big)
	\ar[r, "Bj^*", "\cong"']
	\ar[d, "\tilde{\tau}"]
&
H^3(B\L\SO(d), \R)
	\ar[d, "\tilde{\tau}"]
&
H^4(B\SO(d), \R)
\ar[l, "\tau"', "\cong"]
	\ar[d, "\tilde{\tau}"]
\\
H^2(\Aut(A_d), \R)
	\ar[r, "\theta^*"]
&
H^2\big(\O_{\res}\big(H^d\big), \R\big)
	\ar[r, "j^*", "\cong"']
&
H^2(\L \SO(d), \R)
&
H^3(\SO(d), \R),
\ar[l, "\tau"']
\end{tikzcd}
\end{equation*}
which commutes by naturality of the transgression maps \eqref{TildeTransgression}.
By the results of Section~\ref{sec4} and the universal coefficient theorem, all groups in this diagram are isomorphic to $\R$ and all maps are isomorphisms.

In a first step, we observe that the pullback of the bundle $G_d$ under the map $\theta$ is precisely the implementer bundle $\Imp \to \O_{\mathrm{res}}\big(H^d\big)$, so that the statement of Proposition~\ref{SignProposition} is equivalent to the equality
\begin{equation*}
	2\cdot j^*c_1(\Imp) = \tau (\tilde{\tau}(p_1)).
\end{equation*}

Since the three groups on the right are all Fr\'echet Lie groups, we may work with de Rham cohomology instead of singular cohomology.
Here the transgression homomorphisms $\tilde{\tau}$ may be described as follows.
Let $\alpha \in H^k_{\mathrm{dR}}(BG)$ and let $\pi^*\alpha \in H^k_{\mathrm{dR}}(EG)$ be its pullback.
Then since~$EG$ is contractible, $\pi^* \alpha = d \beta$ for some $\beta \in \Omega^k(EG)$.
The transgression of $\alpha$ is then defined by
\begin{equation}
\label{DefTransgressionTilde}
	\tilde{\tau}(\alpha) = \iota^*\beta,
\end{equation}
where $\iota \colon \SO(d) \to E\SO(d)$ is the inclusion of a fiber.

The de Rham cohomology groups $H^k_{\mathrm{dR}}(\SO(d))$ are best understood using the isomorphism with the Lie algebra cohomology group $H^k(\mathfrak{so}(d), \R)$, which identifies a Lie algebra $k$-cocycle $\alpha$ with the left-invariant 3-form $\overline{\alpha}$ on $\SO(d)$ that coincides with $\alpha$ at the identity.
We now have the following general statement.

\begin{Lemma}
\label{SecondLemmaA}
We have $\tilde{\tau}(p_1) = \overline{\sigma}$, the left invariant $3$-form corresponding to the Lie algebra cocycle
\begin{equation*}
 \sigma(x, y, z) = \frac{1}{8 \pi^2} \langle x, [y, z] \rangle, \qquad x, y, z \in \mathfrak{so}(d).	
\end{equation*}
\end{Lemma}

The proof uses the theory of Chern and Simons \cite{ChernSimons}, which we recall now.
Let $\omega$ be a connection 1-form on a principal $G$-bundle $E$ over $B$ with curvature $\Omega$.
Let $P \in \mathrm{Sym}^2(\mathfrak{g}^*)^G$ be an invariant polynomial on $\mathfrak{g}$.
Then the corresponding Chern--Simons form is $\mathrm{T}P(\omega) \in H^3(E, \R)$ defined by
\begin{equation*}
 \mathrm{T}P(A) = P(\omega \wedge \Omega) - \frac{1}{6} P(\omega \wedge [\omega, \omega]),
\end{equation*}
see \cite[formula~(3.5)]{ChernSimons}.
Denoting by $\iota \colon G \to E$ the inclusion of a fiber (for some fixed base point~$e \in E$), one uses that the curvature form $\Omega$ is horizontal, so that $\iota^*\Omega = 0$, while $\iota^*\omega = \omega_G$, the Maurer--Cartan form of $G$.
Hence
\begin{equation*}
\iota^*\mathrm{T}P(\omega) = -\frac{1}{6} P(\omega_G \wedge [\omega_G, \omega_G]),
\end{equation*}
which differs from the formula (3.11) in \cite{ChernSimons} by a factor of 2.
Going through the conventions used in \cite{ChernSimons} for the wedge product and commutator of $\mathfrak{g}$-valued differential forms (see \cite[p.~50]{ChernSimons}), one obtains that $\iota^* \mathrm{T}P(\omega)$ is the left-invariant form corresponding to the Lie algebra cocycle
\begin{equation}
\label{DefinitionAlpha}
 \alpha(x, y, z) =	- \frac{1}{3} P(x \otimes [y, z] + y \otimes [z, x] + z \otimes [x, y]).
\end{equation}

\begin{proof}
We take $G = \SO(d)$ and $E = E\SO(d)$, $B = B\SO(d)$, the universal bundles.
Choosing models for these that are infinite-dimensional Fr\'echet manifolds (for example, the infinite Grassmannian and Stiefel manifold), we may choose a connection 1-form $\omega \in H^1(E\SO(d), \mathfrak{so}(d))$.
Setting $P(X \otimes Y) = -\tr(XY)/8\pi^2$, \cite[Proposition~3.2]{ChernSimons} states that
\begin{equation*}
 d \mathrm{T}P(\omega) = P(\Omega \otimes \Omega) = -\frac{1}{8\pi^2}\tr\big(\Omega^2\big).
\end{equation*}
By the usual Chern--Weil formulas for Pontrjagin classes, we see that this equals the pullback~$\pi^*p_1$ of the de Rham representative of the universal first Pontrjagin class along the bundle projection~$\pi\colon E\SO(d) \to B\SO(d)$.

Using the description \eqref{DefTransgressionTilde} of the transgression homomorphism, we see that $\tilde{\tau}(p_1) = \iota^*\mathrm{T}P(\omega)$.
As discussed above, the pullback $\iota^*\mathrm{T}P(\omega)$ is the left invariant differential form that corresponds to the Lie algebra cocycle $\alpha$ given by \eqref{DefinitionAlpha}.
For our particular choice of $P$, we get
\[
\alpha(x, y, z) = \left(-\frac{1}{3}\right) \cdot\left(-\frac{1}{8\pi^2}\right) \tr(x [y, z] + y [z, x] + z [x, y])
= \sigma(x, y, z).\tag*{\qed}
\]\renewcommand{\qed}{}
\end{proof}

\begin{Lemma}
\label{LemmaA3}
We have $\tau(\overline{\sigma}) = -\overline{\omega}/2 \pi$, where $\overline{\omega}$ is the left-invariant $2$-form on $L\SO(d)$ corresponding to the Lie algebra cocycle
\begin{equation*}
 \omega(X, Y) = \frac{1}{2\pi}\int_{S^1} \tr(X(t)Y^\prime(t)), \qquad X, Y \in \L\mathfrak{so}(d).	
\end{equation*}
\end{Lemma}

This result can be found as \cite[Proposition~4.4.4]{PressleySegal}, but with incorrect prefactors and an incomplete proof, which is why we repeat the proof below.

\begin{proof}
Tangent vectors $X$, $Y$ at $\gamma \in L\SO(d)$ can be identified with those elements of $\L \mathrm{Mat}_{d\times d}$ such that $\tilde{X} \in L\mathfrak{so}(d)$, where $\tilde{X}(t) = \gamma(t)^{-1}X(t)$. We now calculate
 \begin{align*}
 \tau(\overline{\sigma})_\gamma(X, Y)
 &= \int_{S^1} \overline{\sigma}(\gamma^\prime(t), X(t), Y(t)) {\rm d}t= \int_{S^1} \sigma\big( \gamma(t)^{-1}\gamma^\prime(t), \gamma(t)^{-1} X(t), \gamma(t)^{-1}Y(t)\big) {\rm d}t\\
 &= \frac{1}{8 \pi^2}\int_{S^1}\tr\big(\gamma(t)^{-1}\gamma^\prime(t) \big[\tilde{X}(t), \tilde{Y}(t)\big]\big) {\rm d}t.
 \end{align*}
 Let $\beta \in \Omega^1(\L\SO(d))$ be given by
 \begin{equation*}
 \beta_\gamma(X) = \int_{S^1} \tr \big( 	\gamma(t)^{-1}\gamma^\prime(t) \tilde{X}(t)\big) {\rm d}t.
 \end{equation*}
For the exterior derivative of $\beta$, we find
 \begin{align*}
 d\beta_\gamma(X, Y) &= \partial_X \{ \beta(Y)\}_\gamma - \partial_Y \{ \beta(X)\}_\gamma - \beta_\gamma([X, Y]) \\
 &= \int_{S^1} \big( \tr \big(\tilde{X}^\prime(t) \tilde{Y}(t)\big) - \tr \big(\tilde{Y}^\prime(t) \tilde{X}(t)\big) - \tr\big(\gamma(t)^{-1}\gamma^\prime(t) [\tilde{X}(t), \tilde{Y}(t)]\big)\big) {\rm d}t
 \end{align*}
 for suitable extensions of $X$ and $Y$ to vector fields on $\L\SO(d)$.
 Integrating by parts, we see that ${\rm d} \beta = - 2 \cdot 2\pi \cdot \overline{\omega} - 8 \pi^2 \cdot \tau(\overline{\sigma})$, hence $2 \pi \cdot \tau(\overline{\sigma})$ and $-\overline{\omega}$ are cohomologous.
\end{proof}

Let $\tilde{G} \to G$ be a central $\U(1)$-extension of a Fr\'echet Lie groups, inducing a Lie algebra homomorphism $\tilde{\mathfrak{g}} \to \mathfrak{g}$.
Choosing a linear section of the Lie algebra homomorphism $\tilde{\mathfrak{g}} \to \mathfrak{g}$ gives an identification $\tilde{\mathfrak{g}} = \mathfrak{g} \oplus \R$.
With respect to this choice, the Lie bracket of $\tilde{\mathfrak{g}}$ is given by
\begin{equation*}
 	[(X, \lambda), (Y, \mu)] = ([X, Y], \Omega(X, Y))
\end{equation*}
for some continuous Lie algebra cocycle 2-cocycle $\Omega$, which represents a class in $H^2_c(\mathfrak{g}, \R)$.
The first Chern class of the principal $\U(1)$-bundle $\tilde{G} \to G$ is then given by
\begin{equation}
\label{ChernClassVsCocycle}
 c_1(\tilde{G}) = \frac{1}{2\pi} \overline{\Omega}, 	
\end{equation}
where $\overline{\Omega}$ denotes the associated left-invariant 2-form, see for example \cite[Proposition~4.5.6]{PressleySegal}.

In the case that $G = \O_{\res}(H, [L])$ and $\tilde{G} = \Imp_L$ for some real Hilbert space with a Lagrangian~$L \subset H_\C$, the relevant Lie algebra cocycle has been computed in several places; see, e.g., \cite[Theorem~6.10]{Araki1}, \cite[Theorem~10.2]{NeebSemibounded} or \cite[Theorem~6]{Ottesen}.
The result is
\begin{equation}
\label{GroupCocycleImp}
 \Omega(X, Y) = \frac{1}{8} \tr(J [J, X] [J, Y]), \qquad X, Y \in \mathfrak{o}_{\res}\big(H^d\big),
\end{equation}
where $\Gamma = {\rm i}(P_{L} - P_{\overline{L}})$ is the complex structure determined by the Lagrangian $L$.

When trying to apply these results to the specific Hilbert space $H^d$ defined in \eqref{DefinitionHd}, we face the difficulty that the subspace $L^d \subset H^d$ defined in \eqref{SubLagrangianLd} is only a sub-Lagrangian.
We deal with this issue as follows:
\begin{enumerate}\itemsep=0pt
\item[(i)]
If $d$ is even we let $K = \big(L^d \oplus \overline{L}^d\big)^\perp$ be the even-dimensional subspace of constant functions and choose a Lagrangian $L_0 \subset K$.
Then $L = L^d + L_0 \subset H^d_\C$ is a Lagrangian.
\item[(ii)]
If $d$ is odd, we let $K = \big(L^d \oplus \overline{L}^d\big)^\perp \oplus \C$ and again choose a Lagrangian $L_0 \subset K$.
Then~${L = L^d + L_0}$ is a Lagrangian in $H^d_\C \oplus \C$.
By definition, the implementer bundle over $\O_{\res}\big(H^d\big)$ is the restriction of the implementer bundle over $\O_{\res}\big(H^d \oplus \R\big)$, hence the group cocycle associated to this extension is the restriction of the cocycle \eqref{GroupCocycleImp} to~$\mathrm{o}_{\res}\big(H^d\big) \subset \mathfrak{o}_{\res}\big(H^d \oplus \R\big)$.
Given an element $g \in L\SO(d)$, the element $j(g)$ acts on~$H^d \oplus \R$ through multiplication by $g$ on the first summand and the identity on the second, and consequently, for $X \in \L \mathfrak{so}(d)$, the operator $j_*X$ acts by multiplication with $X$ on the first summand and by zero on the second.
\end{enumerate}
To have a uniform notation, we write $H$ for either the Hilbert space $H^d$ or for $H^d \oplus \R$ in the case that $d$ is odd and let $L \subset H$ be the Lagrangian described above.

\begin{Lemma}\label{LemmaA4}
 We have $2 \cdot j^*c_1(\Imp) = - \overline{\omega}/2\pi$.
\end{Lemma}

A similar computation, for the Lagrangian \eqref{OtherLagrangian} instead of $L$, can be found in \cite{KristelWaldorf1}; unfortunately, the result is off by a factor of $\pm {\rm i}$.
Proposition~6.7.1 in \cite{PressleySegal} is an analogous result for the basic central extension of restricted unitary group $\U_{\res}(H)$ of a polarized complex Hilbert space~$H$.

\begin{proof}We will establish the cocycle identity
\begin{equation}\label{EquationOfCocycles}
 2 \cdot j^*\Omega = -\omega,	
\end{equation}
which gives the result by \eqref{ChernClassVsCocycle}.
By continuity and bilinearity, it suffices to verify that both Lie algebra cocycles evaluate identically on the specific Lie algebra elements $X,Y \in L\mathfrak{so}(d)$ of the form
\begin{equation*}
	X(t) = a {\rm e}^{-{\rm i}kt}, \qquad Y(t) = b {\rm e}^{-{\rm i}\ell t}, \qquad a, b \in \mathfrak{so}(d), \qquad k, \ell \in \Z.
\end{equation*}
On these elements, the right-hand side of \eqref{EquationOfCocycles} is given by
\begin{align}
 \omega(X, Y) &= \frac{1}{2\pi}\int_0^{2\pi} \tr(X(t) Y^\prime(t)) {\rm d}t
 \nonumber\\
 &= -\frac{{\rm i}\ell}{2\pi}\int_0^{2\pi} \tr(ab) {\rm e}^{-{\rm i}(k+\ell)t} {\rm d}t
 = \begin{cases} -{\rm i}\ell \cdot \tr(ab) & k+\ell = 0, \\ 0 & \text{otherwise}.
 \end{cases}\label{ToCompareWithOmega}
\end{align}

To calculate the left-hand side of \eqref{EquationOfCocycles}, we write
	\begin{equation*}
 j_*X =
 \begin{pmatrix}
 	x^\prime & \overline{x} \\ x & \overline{x}^\prime
 \end{pmatrix},	
 \qquad
 j_*Y =
 \begin{pmatrix}
 	y^\prime & \overline{y} \\ y & \overline{y}^\prime
 \end{pmatrix}	
\end{equation*}
with respect to the decomposition $H_\C = L \oplus \overline{L}$.
In other words, $x^\prime = P_LXP_L$ and $x = P_{\overline{L}}XP_L$ and similarly for $y^\prime$ and $y$, while $\overline{x}$ denotes the conjugation of $x$ by the real structure of $H_\C$.
Then since $j_*X$ and $j_*Y$ are restricted, the off-diagonal entries $x$ and $y$ are Hilbert--Schmidt operators and a straightforward calculation gives
\begin{equation}
\label{ReformulatedOmega}
 \Omega(j_*X, j_*Y) = \frac{{\rm i}}{2} \tr(\overline{x} y - \overline{y}x).	
\end{equation}

Write $V_n = \big\{\xi \otimes {\rm e}^{\rm int} \mid \xi \in \C^d\big\} \subset H^d_\C$.
We observe that $j_*X$ sends $V_n$ to $V_{n+k}$ and $j_*Xj_*Y$ sends $V_n$ to $V_{n-k-\ell}$.
On the other hand, $P_L$ and $P_{\overline{L}}$ preserve the subspaces $V_n$ for $n \neq 0$ and send~$V_0$ to $K$ (where $K = V_0$ if $d$ is even, while $K = V_0 \oplus \C$ if $d$ is odd).
We conclude that both~$\overline{x} y$ and $x \overline{y}$ send $V_n$ to $V_{n-k-\ell}$, respectively $V_{n-k-\ell} \oplus \C$ if $d$ is odd.
Choosing an orthonormal basis adapted to the decomposition $H_\C = \bigoplus_{n \in \Z} V_n$, respectively $H_\C = (\bigoplus_{n \in \Z} V_n) \oplus \C$ to calculate the trace on the right-hand side of \eqref{ReformulatedOmega}, we see that the result can be non-zero only if $k+\ell = 0$.

So suppose now that $k = - \ell$ and let $\xi \otimes {\rm e}^{{\rm int}} \in V_n^\C$, $n \geq 1$. Then we have
\begin{align*}
 \overline{x}y(\xi\otimes \chi_n) &=
 \begin{cases}
 ab\xi \otimes \chi_{n} & 1 \leq n \leq \ell -1, \\
 a\overline{P}_0b\xi \otimes \chi_{n} & n = \ell,\\
 0 & \text{otherwise},
 \end{cases}
 \\
 \overline{y}x(\xi \otimes \chi_n) &=
 \begin{cases}
 ba\xi \otimes \chi_{n}, & 1 \leq n -1 \leq -\ell -1, \\
 b\overline{P}_0a\xi \otimes \chi_{n}, & n = - \ell,\\
 0, & \text{otherwise},
 \end{cases}	
\end{align*}
where $P_0$ is the orthogonal projection onto $L_0$ in $K$.
For $\xi \otimes 1 \in V_0$, we find
\begin{align*}
 \overline{x}y(\xi\otimes 1) &=
 \begin{cases}
 P_0ab\xi \otimes 1, & \ell > 0, \\
 P_0a\overline{P}_0b\xi \otimes 1, & \ell = 0,\\
 0, & \text{otherwise},
 \end{cases}
 \\
 \overline{y}x(\xi \otimes 1) &=
 \begin{cases}
 P_0ba\xi \otimes 1, & \ell < 0, \\
 P_0b\overline{P}_0a\xi \otimes 1, & \ell = 0,\\
 0, & \text{otherwise}.
 \end{cases}	
\end{align*}
Concluding, if $\ell >0$, we obtain
\begin{equation*}
 \Omega(j_*X, j_*Y) = \frac{{\rm i}}{2} \tr(\overline{x}y) = \frac{{\rm i}}{2} \left(\sum_{n=1}^{\ell-1}\tr(ab) + \tr\big(a\overline{P}_0 b\big) + \tr(P_0ab) \right) = \frac{{\rm i}\ell}{2}\tr(ab).
\end{equation*}
Here we used that
\[\tr\big(a\overline{P}_0 b\big) + \tr(P_0ab) = \tr(ab),\] which follows from the following calculation.
First, because $a$ and $b$ are real and skew-adjoint and $P_0$ is self-adjoint, we get
\begin{equation*}
 	\tr(aP_0b) = \tr(baP_0) = \overline{\tr((baP_0)^*)} = \overline{\tr(P_0 (-a)(-b))} = \tr\big(\overline{P}_0ab\big),
\end{equation*}
hence
\begin{equation*}
	\tr\big(a\overline{P}_0 b\big) + \tr(P_0ab) = \tr\big(\overline{P}_0 ab\big) + \tr(P_0ab) = \tr(ab),
\end{equation*}
using $\overline{P}_0 + P_0 = \id$.
If $\ell < 0$, we obtain in a similar fashion
\begin{equation*}
 \Omega(j_*X, j_*Y) = -\frac{{\rm i}}{2} \tr(\overline{y}x) = -\frac{{\rm i}}{2} \left(\sum_{n=1}^{|\ell|-1}\tr(ba) + \tr\big(b\overline{P}_0 a\big) + \tr(P_0ba) \right) = -\frac{{\rm i}|\ell|}{2}\tr(ab).
\end{equation*}
Finally, if $\ell=0$, then
\begin{equation*}
\begin{aligned}
 	\Omega(j_*X, j_*Y) = \frac{{\rm i}}{2} \tr(\overline{x}y - \overline{y}x)
 	&= \frac{{\rm i}}{2} \big(\tr\big(P_0 a\overline{P}_0b\big) - \tr\big(P_0 b\overline{P}_0a\big)\big)\\
 	&= \frac{{\rm i}}{2} \big(\tr(P_0 ab) - \tr(P_0 a{P}_0b) - \tr(P_0 ba) + \tr(P_0 bP_0a)\big) \\
 	&=\frac{{\rm i}}{2} \big(\tr(P_0 ab) - \tr(ba) + \tr\big(\overline{P}_0 ba\big) \big) \\
 	& = 0.
\end{aligned}
\end{equation*}
Comparing with \eqref{ToCompareWithOmega}, this establishes \eqref{EquationOfCocycles}.
\end{proof}

\subsection*{Acknowledgements}
It is a pleasure to thank Achim Krause, David Roberts, Raphael Schmidpeter and Konrad Waldorf for helpful discussions.
I would also thank the anonymous referees for their suggestions, which improved the paper.
I am indebted to SFB 1085 ``Higher Invariants'' for financial support.

\pdfbookmark[1]{References}{ref}
\LastPageEnding

\end{document}